\documentclass{amsart}

 \pdfoutput=1
\usepackage{amssymb,amsmath,amsthm,enumitem}
\usepackage[usenames,dvipsnames]{xcolor}
\usepackage{bbm}
\usepackage[all]{xy}  
\newcommand{\comment}[1]{}

\newcommand{\C}{{\mathbb C}}

\newcommand{\E}{{\mathbb E}}
\renewcommand{\P}{{\mathbb P}} 
\newcommand{\F}{{\mathcal F}}
\newcommand{\CF}{{\mathcal F}}
\newcommand{\CE}{{\mathcal E}}
\newcommand{\CT}{{\mathcal T}}
\newcommand{\CP}{\mathcal{P}}

\newcommand{\lam}{{\lambda}}
\newcommand{\eps}{{\epsilon}}

\newcommand{\Tr}{\operatorname{Tr}}
\newcommand{\conj}[1]{\overline{#1}}
\newcommand{\diag}{\operatorname{diag}}

\newcommand{\rk}{\operatorname{rk}}
\newcommand{\SC}{\operatorname{SC}}
\newcommand{\BC}{\operatorname{BC}}

\newcommand{\block}[2]{\left[\begin{array}{#1}#2\end{array}\right]}
\newcommand{\eql}[1]{\begin{equation}#1\end{equation}}
\newcommand{\eq}[1]{\begin{equation*}#1\end{equation*}}
\newcommand{\al}[1]{\begin{align*}#1\end{align*}}
\newcommand{\all}[1]{\begin{align}#1\end{align}}

\newtheorem{lem}{Lemma}
\newtheorem{thm}{Theorem}
\newtheorem{rem}{Remark}
\newtheorem{prop}{Proposition}
\newtheorem{defn}{Definition}

\usepackage{graphicx}
\usepackage{caption}
\usepackage{adjustbox}
\usepackage[pdftex]{hyperref}

\begin{document}
\title{Outlier eigenvalue fluctuations of perturbed iid matrices}
\author{Anand B. Rajagopalan}
\address{Department of Mathematics, UCLA, Los Angeles CA 90095-1555}
\email{anandbr@math.ucla.edu}
\begin{abstract}  It is known that in various random matrix models, large perturbations create outlier eigenvalues which lie, asymptotically, in the complement of the support of the limiting spectral density. This paper is concerned with fluctuations of these outlier eigenvalues of iid matrices $X_n$ under bounded rank and bounded operator norm perturbations $A_n$, namely with $\lam(\frac{X_n}{\sqrt{n}}+A_n)-\lam(A_n)$. The perturbations we consider are allowed to be of arbitrary Jordan type and have (left and right) eigenvectors satisfying a mild condition. We obtain the joint convergence of the (normalized) asymptotic fluctuations of the outlier eigenvalues in this setting with a unified approach. 
\end{abstract}

\maketitle

\section{Introduction}
\subsection{Background}
Following the works of \cite{bbp} and \cite{bsspike} investigating the asymptotic spectrum of perturbed empirical covariance matrices or spiked population models, various efforts have been undertaken to better understanding the outlier eigenvalues of perturbed random matrix models. In the Hermitian setting, the works of \cite{cdf1}, \cite{cdf2},\cite{prs},\cite{rs},\cite{ky1}, and \cite{ky2} build up to an essentially complete picture of the asymptotic locations and normalized fluctuations of the outlier eigenvalues of bounded rank and bounded operator norm perturbations. 

This paper obtains the asymptotic fluctuations of outlier eigenvalues for the iid matrix ensemble under the same class of perturbations. Before stating our results, we introduce the theorem on the asymptotic location of the outlier eigenvalues due to \cite{tao} after presenting some introductory definitions and results.

\begin{defn}
A \emph{iid matrix} $X$ is an infinite array of (complex) iid random variables $(x_{i,j})_{i,j\geq 1}$ which we identify with the sequence $(X_n)_{n\geq 1}$, $X_n=(x_{i,j})_{1\leq i,j\leq n}$. We assume that the \emph{atom distribution} $x=x_{1,1}$ satisfies the moment conditions $\E x=0$ and $\E|x|^2=1$. We let $\Lambda(Y)$ denote the spectrum of $Y$ and let 
\eq{\mu_n:=\frac{1}{n}\sum_{\lam\in\Lambda(X_n)}\delta_{\frac{\lam}{\sqrt{n}}}} denote the \emph{empirical spectral distribution} of $X$.
\end{defn}
\begin{thm}[Circular law]
For an iid matrix $X$, we have
\eq{\mu_{X_n}\Rightarrow \mu_C:=\frac{1}{\pi}\mathbbm{1}_{\{z\in\C:|z|\leq 1\}}} almost surely, where $\Rightarrow$ denotes weak convergence.
\end{thm}
The circular law, which is the work of many authors (see \cite{taocirc} and references therein), in particular implies that the spectral radius of $X/\sqrt{n}$, $\rho(X/\sqrt{n})$, satisfies $\limsup\rho(X/\sqrt{n})\geq 1$ almost surely. The following is a complementary result; see \cite{bai-yin} for a proof. 
\begin{thm}\label{opnorm} Let $X_n$ be an iid matrix with atom distribution having bounded fourth moment. Then \eq{\rho\left(\frac{X}{\sqrt{n}}\right)=\lim_{l\rightarrow\infty}\left\|\left(\frac{X}{\sqrt{n}}\right)^l\right\|^{1/l}} converges to $1$ almost surely as $n\rightarrow\infty$. Moreover, for $l\geq 1$, $\|(\frac{X}{\sqrt{n}})^l\|$ converges to $l+1$ almost surely as $n\rightarrow\infty$. 
\end{thm}

Now let $A=A_n$ be a deterministic matrix of rank $O(1)$ and operator norm $O(1)$. We will assume for notational convenience that 
$\Theta=\Theta_n:=\{\lam\in\Lambda(A_n):|\lam|>1\}$ is independent of $n$ for $n$ sufficiently large and we let $m_{\theta}$ denote the multiplicity of $\theta$. Then the following theorem (due to \cite{tao}, with generalizations to other models in \cite{renfrew}, \cite{rochet} and \cite{bc}) shows that outliers in the spectrum of $\frac{X}{\sqrt{n}}+A$ appear, in contrast to the situation in Theorem~\ref{opnorm}.

\begin{thm}\label{taothm}
Let $X$ be an iid matrix with bounded fourth moment and let $A$ and $\Theta$ be as above. For each $\theta\in\Theta$ there exists
\eq{\Lambda^{\theta}\subset\Lambda(\frac{X}{\sqrt{n}}+A)} with $|\Lambda^{\theta}|=m_{\theta}$ and for $\lam\in\Lambda^{\theta}$,
\eq{\lambda\rightarrow \theta} almost surely.
\end{thm}
\begin{figure}
    \centering
    \includegraphics[scale=0.5]{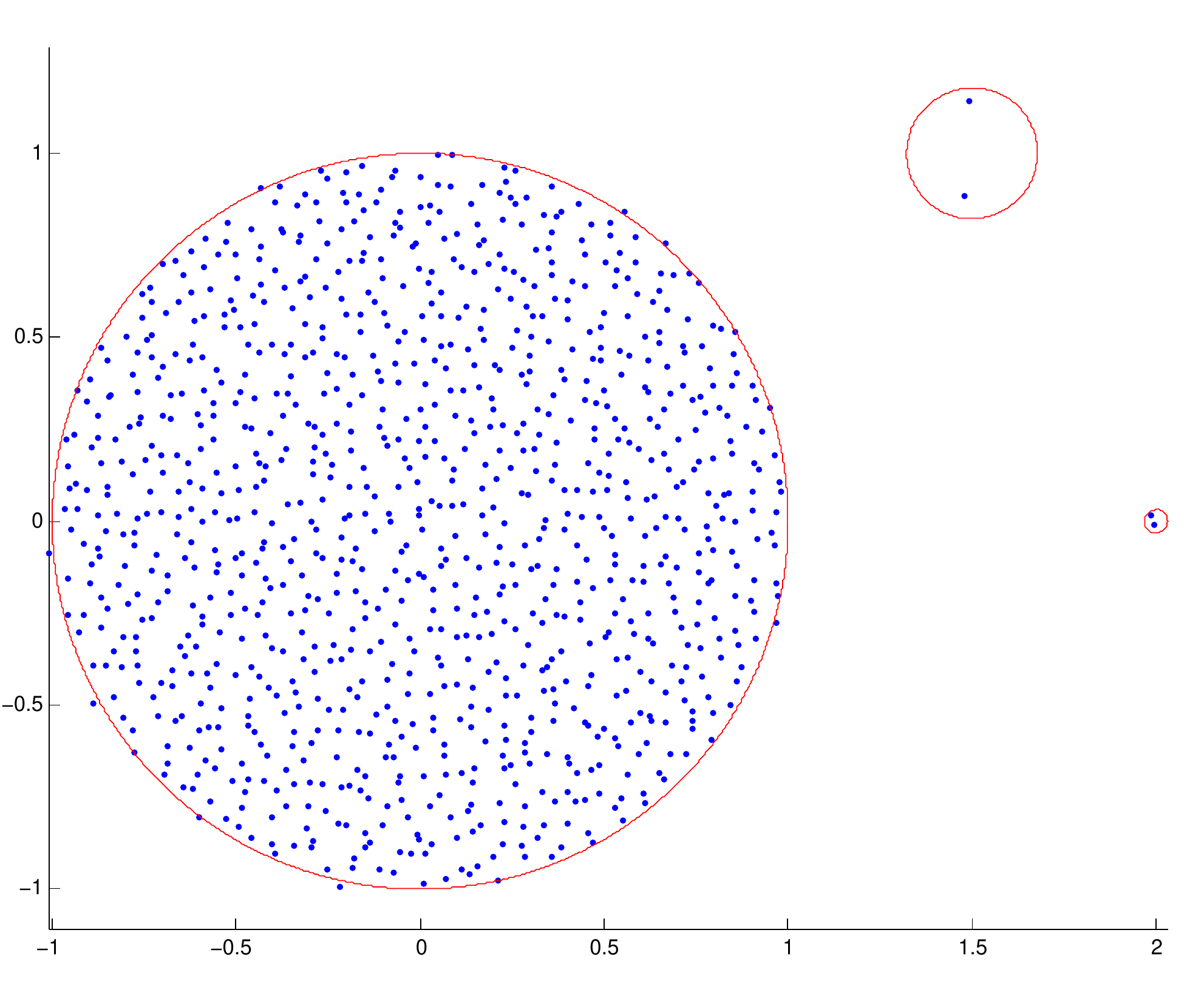}
    \caption{Eigenvalues of $X/\sqrt{n}+A$ with $X$ having iid $\mathcal{N}(0,1)_{\C}$ entries, $A=2I_2\oplus J_{1.5+i,2}\oplus 0_{996}$ and $n=1000$. The smaller circles are of radii $n^{-1/2}$ and $n^{-1/4}$.}
    \label{macrogauss}
\end{figure}
 
To illustrate Theorem~\ref{taothm}, in Figure~\ref{macrogauss} we have plotted the eigenvalues of a perturbed Gaussian matrix $X/\sqrt{n}+A$, with $x$ having distribution $\mathcal{N}(0,1)_{\C}$ and $n=1000$. The two outliers near $2$ correspond to the block $\left(\begin{smallmatrix}2&0\\0&2\end{smallmatrix}\right)$ and the two outliers near $1.5+i$ are from the block $\left(\begin{smallmatrix}1.5+i&1\\0&1.5+i\end{smallmatrix}\right)$ of $A$. Observe that the fluctuations from the Jordan block are larger; this phenomenon will be discussed later.

\subsection{Model and statement of results}
The focus of our paper is the fluctuations $\lam-\theta$. More precisely, we obtain the limiting distribution of the normalized fluctuations when $A$ is allowed to have arbitrary Jordan type and under certain sparsitiy and uniformity assumptions on the (left and right) eigenvectors of $A$. After introducing the main definition and theorem in this subsection, we will discuss simpler special cases in Subsection~\ref{ssec-examples}.

We now define the perturbation matrices we will consider in this paper, along with associated notation. To unify notation in this paper, for any complex vector $z$, we let \eql{\label{conjdef}z^{(d)}:=\left\{
     \begin{array}{lr}
       z & : d=0\\
       \conj{z} & : d=1
     \end{array}
   \right.} where $\conj{z}$ denotes the (componentwise) conjugate of $z$. We will write $z^T$ for the transpose of $z$ and $z^*$ for the conjugate transpose of $z$.   
\begin{defn}
A \emph{perturbation matrix} $A=(A_n)_{n\geq 1}$ is a sequence of (complex) $n\times n$ matrices with rank $O(1)$ and operator norm $O(1)$. For $\theta\in\Theta=\{\theta\in\Lambda(A_n):|\theta|>1\}$, let $J_{\theta}$ be the Jordan block in the Jordan decomposition of $A$ corresponding to $\theta$ with blocks written in nonincreasing order. We will assume that $\Theta$ and $(J_{\theta})_{\theta\in\Theta}$ are independent of $n$ for $n$ sufficiently large. Let
\eq{J_{\theta}=\bigoplus_{k=1}^{K_{\theta}}J_{\theta,k}^{m_{\theta,k}} \text{ where } 
J_{\theta,k}:=\left(\begin{smallmatrix}\theta&1&&\\&\theta&1&\\&& \ddots&1\\&&&\theta\end{smallmatrix}\right)} is the Jordan block of size $k$ occuring with multiplicity $m_{\theta,k}$ in $J_{\theta}$. To index the eigenvectors and generalized eigenvectors, we introduce the following notation. Let
\eq{I:=\{s=(i,j,k,\theta):i\in[k],j\in[m_{\theta,k}],k\in[K_{\theta}],\theta\in\Theta\}} and for $s\in I$, we write $s=(i_s,j_s,k_s,\theta_s)$. Let
\eq{I_{\theta}=\{s\in I:\theta_s=\theta\}.} For fixed $j$, $k$ and $\theta$, let $(v_s)_{i=1}^k$ be the generalized eigenvectors corresponding to the $j$th block of $J_{\theta,k}$, and let $v_{1,j,k}$ be the eigenvector for that block. Similarly define $(u^*_{s})_{i=1}^k$ to be the generalized left eigenvectors with the $u^*_{(k,j,k)}$'s being the left eigenvectors. To index the left and right eigenvectors, we let 
\eq{I^{\theta}_u:=\{s\in I_{\theta}:i_s=k_s\}} and
\eq{I^{\theta}_v:=\{t\in I_{\theta}:i_t=1\}.} Finally, we let
\eq{I_2:=\bigcup_{\theta\in\Theta}I^{\theta}_u\times I^{\theta}_v\times\{\theta\}} and for $r\in I_2$, we write $r=(s_r,t_r,\theta_r)$.

For $(s_i,t_i,\theta_i)\in I_2$, $i=1,2$, we assume that the limits of the following inner products exist and define, for $d_1, d_2\in \{0,1\}$, the scalars
\eql{\label{lim1}U^{(d_1),(d_2)}_{s_1,s_2}:=\lim_{n\to\infty}(u_{s_1})^{(d_1)*}\conj{(u_{s_2})}^{(d_2)},}
\eql{\label{lim2}V^{(d_1),(d_2)}_{t_1,t_2}:=\lim_{n\to\infty}(v_{t_1})^{(d_1)T}(v_{t_2})^{(d_2)}.} We also assume the following convergence and define $(G_r)_{r\in I_2}$ by
\eql{\label{grdef}(u_{s_r}^*Xv_{t_r})_{r\in I_2}\Rightarrow (G_r)_{r\in I_2}.} Lastly, we require the following technical assumption. Fix $\delta>0$ and let 
\eq{L=\bigcup_{r\in I_2}\{(i,j)\in[n]^2:|u_{s_r,i}v_{t_r,j}|\geq n^{-1/4+\delta}\}.} Then we assume 
\eql{\label{technical}\left(\sum_{(i,j)\in L}u_{s_r,i}x_{ij}v_{t_r,j}\right)_{r\in I_2}\Rightarrow (G^L_r)_{r\in I_2}.}
\end{defn}

\begin{rem}\label{rem-conv2}
The eigenvectors satisfying the convergence criteria of ~\eqref{lim1}-~\eqref{technical} are quite general, and are allowed to be of \emph{local}, \emph{delocal} and \emph{mixed} types (see Remark~\ref{rem-conv}). These eigenvector requirements are similar to those of \cite{ky1} and \cite{ky2}.
\end{rem} 

We denote the Schur complement of $A$ in the block matrix $\left(\begin{smallmatrix}A&B\\C&D\end{smallmatrix}\right)$ by \eq{\SC(A,\left(\begin{smallmatrix}A&B\\C&D\end{smallmatrix}\right)):=D-CA^{-1}B.} Recalling the notation of Theorem~\ref{taothm}, we denote the elements of $\Lambda^{\theta}$ by $\lam^{\theta}_s$ for $s\in I_{\theta}$. 
We now state our main theorem.
\begin{thm}\label{delocthm}
Let $X$ be an iid matrix and $A$ a perturbation matrix. We will assume the moment hypothesis
$\E|x|^m<\infty$, with $m$ defined as follows. First define $c$ through 
\eql{\label{cdef}c=\sup\{c'\geq 0:\max_{i\in[p]}\|u_i\|_{\infty}\|v_i\|_{\infty}\ll n^{-c'}\}.} Then fix $\eps>0$ and set \eql{\label{momhyp}m=\min(\max(2/c,4),8)+\eps.}
Recalling ~\eqref{grdef}, ~\eqref{lim1} and ~\eqref{lim2}, we define the random variables $(F_r)_{r\in I_2}$ by 
\eql{\label{Ftheta}F_r:=G_r+g_r,} where
$(g_r)_{r\in I_2}$ is a collection of centered complex Gaussians independent of $(G_r)_{r\in I_2}$ with mixed second moments specified by
\eql{\label{gtheta}\E g^{(d_1)}_{r_1}g^{(d_2)}_{r_2}=
\frac{(\E x^{(d_1)}x^{(d_2)})^2}{\theta_{r_1}\theta_{r_2}-\E x^{(d_1)}x^{(d_2)}}U^{(d_1),(d_2)}_{s_{r_1},s_{r_2}}V^{(d_1),(d_2)}_{t_{r_1},t_{r_2}}.} For $\theta\in\Theta$, let $F^{\theta}:=(F_r)_{\theta_r=\theta}$ be the $I^{\theta}_u\times I^{\theta}_v$ matrix of random variables and for $k\in[K_{\theta}]$, let 
\eq{F^{\theta,k}:=\SC(F^{\theta}|_{\{(s,t):k_s,k_t\geq k+1\}}, F^{\theta}|_{\{(s,t):k_s,k_t\geq k\}})} be the $m_{\theta,k}\times m_{\theta,k}$ matrix that is the Schur complement of the indicated submatrices of $F^{\theta}$. Denote the eigenvalues of $F^{\theta,k}$ by $(\tilde{\lam}^{\theta}_{j,k})_{j=1}^{m_{\theta,k}}$ whose $k$th roots we denote
\eql{\label{thyfluc}\tilde{f}^{\theta}_{i,j,k}:=(\zeta^i_k(\tilde{\lam}^{\theta}_{j,k})^{1/k})_{(i,j,k)\in I_{\theta},\theta\in\Theta}} where $\zeta_k=e^{\frac{2\pi\sqrt{-1}}{k}}$. Then for each $\theta\in\Theta$, we can label the eigenvalues in $\Lambda^{\theta}$ as $(\lam^{\theta}_{i,j,k})_{(i,j,k)\in I_{\theta}}$ such that the normalized outlier fluctuations 
\eql{\label{realfluc}f^{\theta}_{i,j,k}:=n^{1/(2k)}\left(\lam^{\theta}_{i,j,k}\left(\frac{X}{\sqrt{n}}+A\right)-\theta\right)} converge to $(\tilde{f}^{\theta}_{i,j,k})_{\theta\in\Theta,(i,j,k)\in I_{\theta}}$ in the following sense. Define the subgroup $S$ of the permutation group $S_I$ by
\al{S:=&\{\pi\in S_I: \pi(s)_{\theta}=s_{\theta},\pi(s)_k=s_k\text{ and }\\
&\pi(s)_j=\pi(t)_j\Leftrightarrow s_j=t_j\text{ for all }s,t\in I\}.} Let $\BC(\C^I)^S$ denote the set of bounded continuous functions on $\C^I$ invariant under the action of $S$. Then for $f\in\BC(C^I)^S$, and writing $(\tilde{f}_l)_{l\in I}$ for ~\eqref{thyfluc} and $(f_l)_{l\in I}$ for ~\eqref{realfluc}, 
\eq{\int fd\mu_{(f_l)_{l\in I}}\rightarrow \int fd\mu_{(\tilde{f}_l)_{l\in I}}.}
 
\end{thm}
\begin{rem}
The moment hypothesis we require seems to be a technical limitation of the moment method that we have employed. While we need at most $8+\eps$ moments in all cases, we conjecture that $4$ moments always suffice. In the \emph{delocal} case with $c=1$ (i.e., $\|u_i\|_{\infty},\|v_i\|_{\infty}\ll 1/\sqrt{n}$, we require $4+\eps$ moments which almost matches the conjectured optimal. On the other hand, under the assumption of $4$ moments, \cite{bc} obtains the fluctuations of certain types of \emph{local} matrices (with $c=0$) as described in the next subsection.  
\end{rem}

\subsection{Discussion and related works}\label{ssec-examples}

We now provide examples of different types of behavior for the fluctuations that illustrate Theorem~\ref{delocthm}. The first two examples are of rank $1$ fluctuations.
\begin{enumerate}[label=(\roman*)]
\item If $A$ is has a single non zero entry $\theta$ in the top left with $|\theta|>1$, the limiting normalized  fluctuation of the outlier is the law of $x+g$ where $x$ is the atom distribution and $g$ is a centered complex Gaussian with $\E g^2=0$ and $\E|g|^2=\mathcal{N}(0,\frac{1}{|\theta|^2-1})_{\C}$. In Figure ~\ref{unifgausscomp}, we demonstrate this non-universality in the case $\theta=2$ and $x$ as specified in the captions. 

\item If $A=\theta vu^*$ is of rank $1$ with $|\theta|>1$ and $\|u\|_{\infty}\|v\|_{\infty}=o(1)$, then the normalized fluctuation $\sqrt{n}(\lam-\theta)$ converges to the law of a centered complex Gaussian $g_{\theta}$ with 
\eq{\E g^2_{\theta}=\frac{|\theta|^2\E x^2}{|\theta|^2-\E x^2}\lim_{n\rightarrow\infty}u^*\conj{u}v^Tv} and
\eq{\E |g_{\theta}|^2=\frac{|\theta|^2}{|\theta|^2-1}\lim_{n\rightarrow\infty}u^*uv^*v.} In particular, if $\E x^2=0$ and $A$ is normal (thus $u$ and $v$ are unit vectors), then $g_{\theta}$ is a circularly symmetric Gaussian with variance $\frac{|\theta|^2}{|\theta|^2-1}$.

\item Suppose $A = UDU^*$ is normal of rank $k$, with $\|u_i\|_{\infty}=o(1)$ for $i=1,2,\ldots,k$. For a fixed eigenvalue $\theta\in\Theta$ of multiplicity $m$, the covariance formula ~\eqref{gtheta} reduces to
  \eq{\E g_{ab}g_{cd}=\frac{\theta^2\E x^2}{\theta^2-\E x^2}\lim_{n\rightarrow\infty}u^*_a\conj{u_c}{u^T_bu_d}}
and
  \al{\E g_{ab}\conj{g_{cd}}&=\frac{\theta^2}{\theta^2-1}\lim_{n\rightarrow\infty}u^*_au_cu^*_bu_c\\
	&=\frac{\theta^2}{\theta^2-1}\delta_{ac}\delta_{bd}}
	
Note that fluctuations of different eigenvalues are still correlated in general. We obtain asymptotically independent fluctuations for distinct eigenvalues in the following cases.
\begin{enumerate}
\item If $A$ is real, $u^*_a\conj{u_c}=\delta_{ac}$, ${u^T_bu_d}=\delta_{bd}$ and the entries of $F^{\theta}=(g_{ab})_{a,b=1}^m$ are independent Gaussians. Depending on the Jordan structure $J_{\theta}$, the normalized fluctuations converge to the appropriate roots of eigenvalues of Schur complements of submatrices of $F^{\theta}$ as specified in Theorem~\ref{delocthm}. 

\item If $\E x^2=0$, $F^{\theta}=(g_{ab})_{a,b=1}^m$ is a scaled complex Ginibre ensemble with atom distribution $g$ satisfying $\E g=0$, $\E g^2=0$ and $\E |g|^2=\frac{|\theta^2|}{|\theta^2|-1}$. If we now suppose further that $J_{\theta}=\theta I_m$, then the $m$ fluctuations associated to $\theta$ are given by the eigenvalues of the complex Ginibre ensemble specified above. By the circular law, they lie approximately uniformly in a disk of radius $\frac{|\theta|}{(|\theta|^2-1)^{\frac{1}{2}}}$ for $m$ large.  

\item So far, the fluctuations have been of order $O(\frac{1}{\sqrt{n}})$. Suppose again that $\E x^2=0$ but that $J_{\theta}$ is a single Jordan block of size $m$. Then as remarked below Proposition ~\ref{detpert}, the $m$ fluctuations scaled by $n^{1/(2m)}$ are given by $(e^{2\pi ij/m}g^{1/m}_{\theta})_{j=0}^{m-1}$ where $g_{\theta}=(F^{\theta})_{m1}$ is the lower left entry of $F^{\theta}$. Hence the fluctuations are distributed uniformly around a circle of radius $n^{-1/(2r)}g^{1/m}_{\theta}$. This dependence of the rate of convergence on the size of the Jordan block is illustrated by the outliers in Figure~\ref{macrogauss}. 

\end{enumerate} 

\end{enumerate}

\begin{figure}
\centering
\adjustbox{valign=t}{
\begin{minipage}{.5\textwidth}
  \centering
  \includegraphics[width=\linewidth]{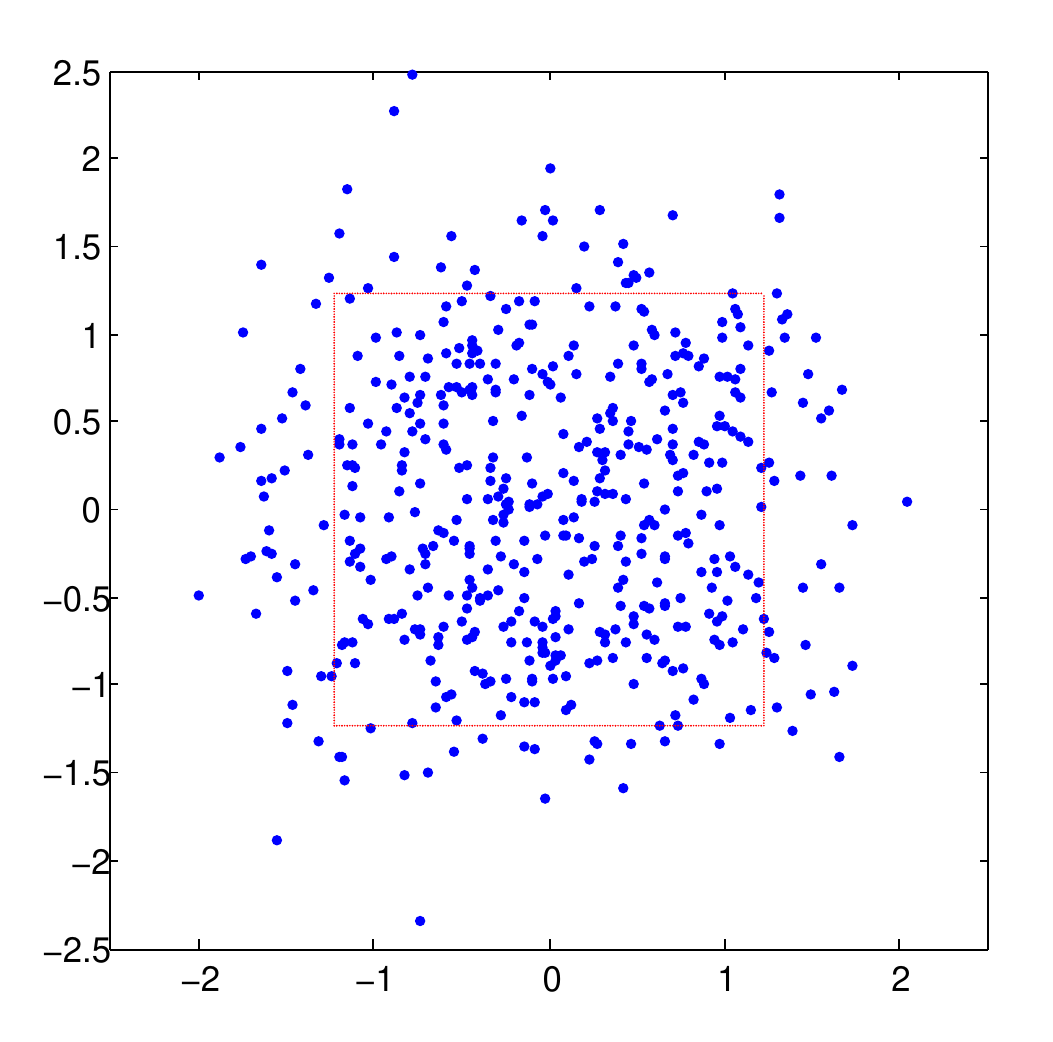}
	\caption*{(2a)}
  %\captionof*{figure}{$x$ is distributed uniformly over the square $[-l,l]^2\subset\C$ with $l=\sqrt{3/8}$, outlined above.}
  \label{fig-unif1}
	\includegraphics[width=\linewidth]{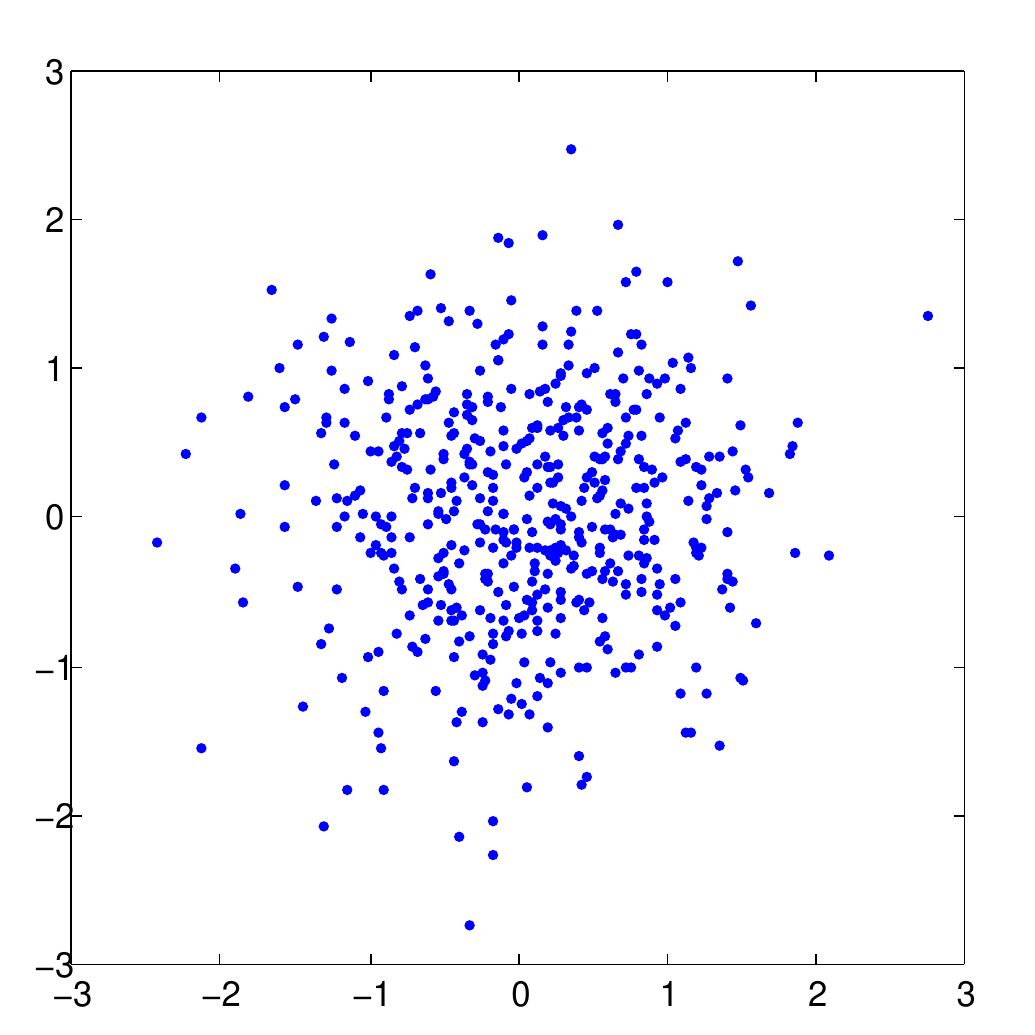}
	\caption*{(2c)}
  %\captionof*{figure}{$x$ is distributed uniformly over the square $[-l,l]^2\subset\C$ with $l=\sqrt{3/8}$, outlined above.}
  \label{fig-gauss1}
\end{minipage}%
}%
\adjustbox{valign=t}{
\begin{minipage}{.5\textwidth}
  \centering
  \includegraphics[width=\linewidth]{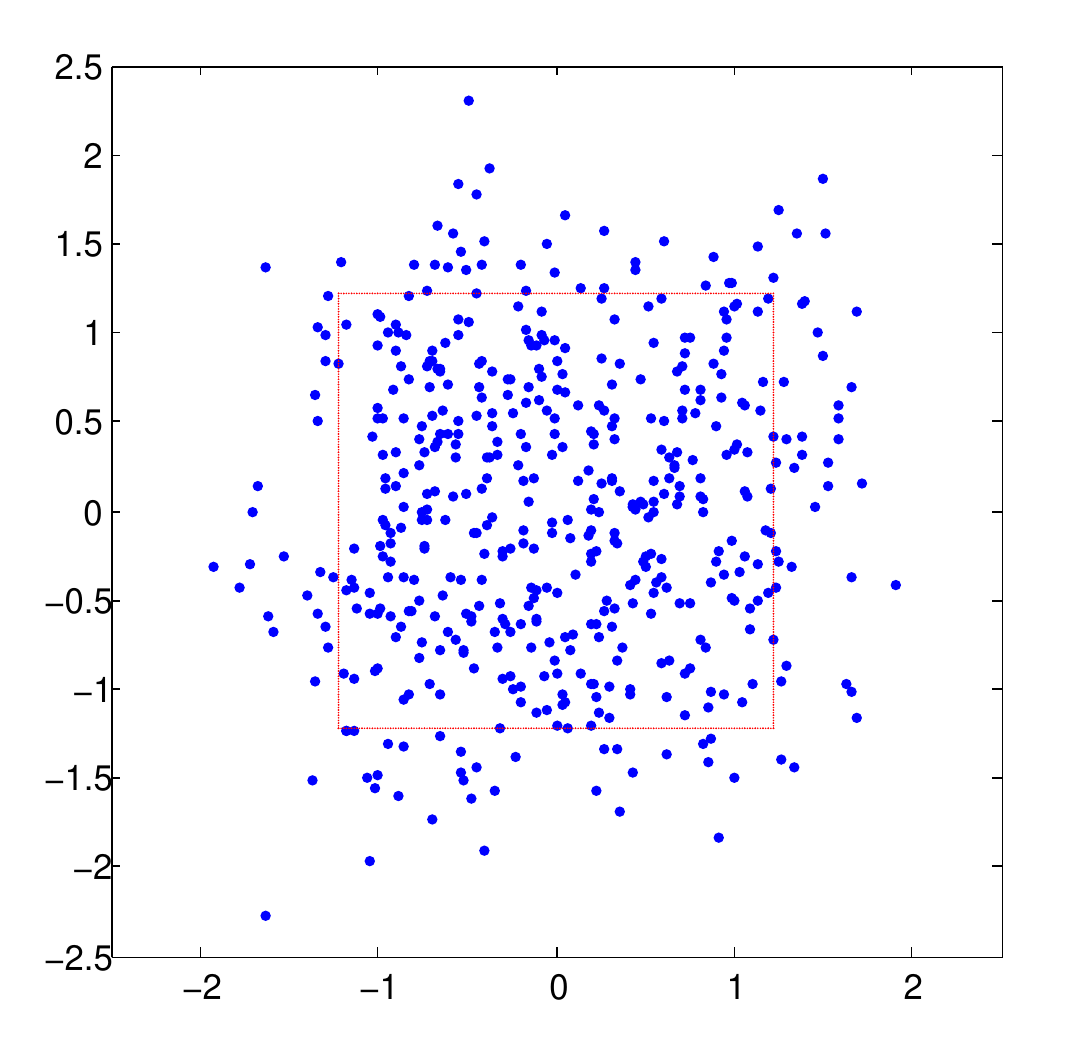}
	\caption*{(2b)}
  %\captionof*{figure}{A sample of size $500$ from the theoretical limiting distribution}
  \label{fig-unif2}
	\vspace{2.5mm}

	\includegraphics[width=\linewidth]{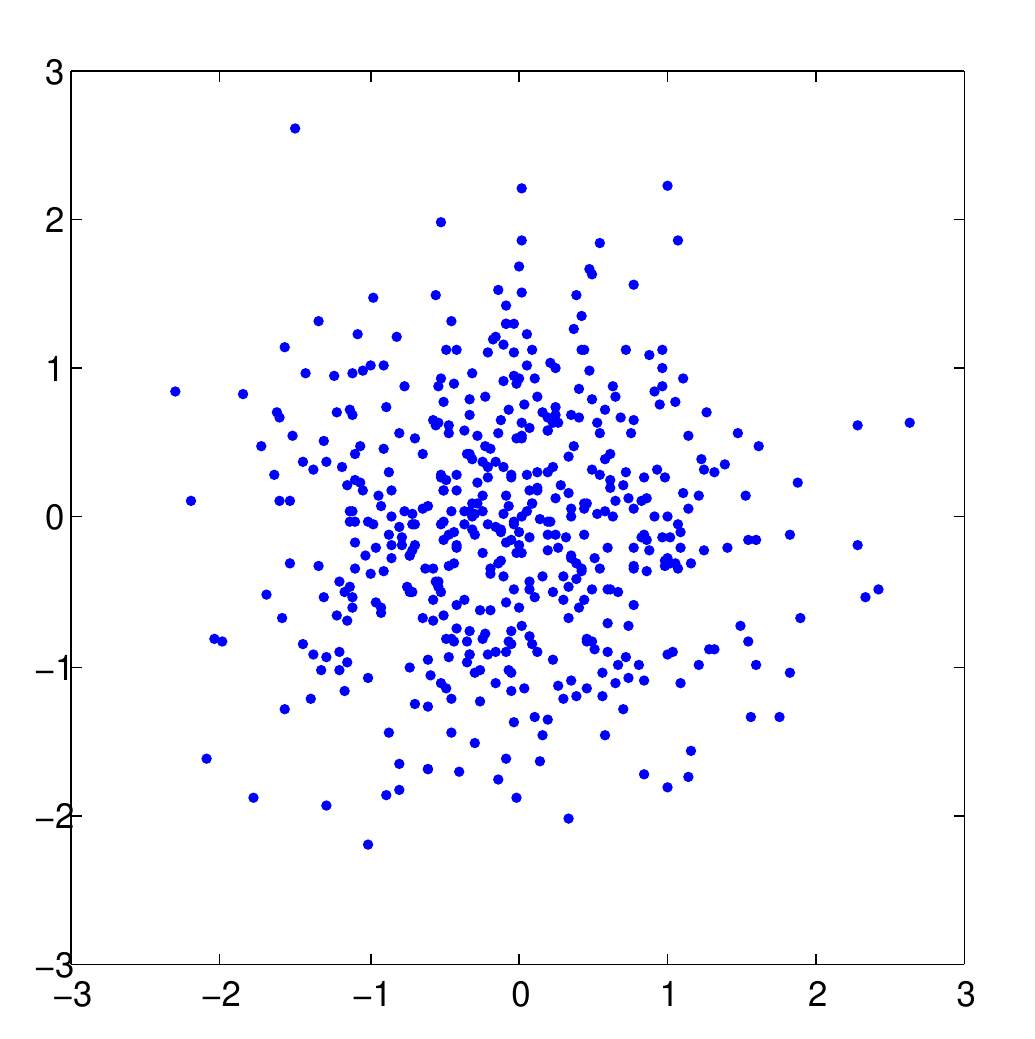}
		\caption*{(2d)}
  %\captionof*{figure}{A sample of size $500$ from the theoretical limiting distribution}
  \label{fig-gauss2}
\end{minipage}
}
%$x$ is distributed as $\mathcal{N}(0,1)_{\C}$, that is $\Re(x)$ and $\Im(x)$ are independent and have variance $1/2$ each.

\caption{ Figures $2$a and $2$c are $500$ samples of the normalized fluctuations $\sqrt{n}(\lam_{out}(\frac{X}{\sqrt{n}}+A)-2)$ of a single outlier with $n=100$ and $A$ given by $a_{i,j}=2\delta_{(i,j)=(1,1)}$. In Figure $2$a, the atom distribution $x$ is distributed uniformly over the square $[-l,l]^2\subset\C$ with $l=\sqrt{3/2}$ so that $\E|x|^2=1$ (outlined in figure). In Figure $2$c, $x$ is the standard complex normal $\mathcal{N}(0,1)_{\C}$. Figures $2$b and $2$d are $500$ samples from the corresponding limiting distributions as predicted by Theorem~\ref{delocthm} and detailed in case (i).\label{unifgausscomp}}
\end{figure}

In \cite{rochet}, the outlier eigenvalues of perturbations of the single ring model are studied and their locations and limiting fluctuations are obtained (\cite[Theorem 2.9]{rochet}) for finite rank and finite operator norm perturbations of arbitrary Jordan type. Note that the special case of the Ginibre ensemble, which is an iid matrix, is contained in this model as well. Our approach to dealing with perturbations of various Jordan types is similar and relies on a deterministic perturbation result known as the Lidskii-Vishik-Lyusternik perturbation theorem (see \cite{lidskiiorig}, \cite{ljusternik}, \cite{lidskii} and references therein) which we have reproduced in Appendix~\ref{secdet}.

In \cite{bc}, Bordenave and Captaine study asymptotic outlier locations and fluctuations for perturbed iid matrices. The perturbations considered there are of the form $A=A'+A''$ where $A''$ is of bounded rank and $A'$ (with possibly unbounded rank) satisfies a well-conditioning property. In the case of local perturbations, where $A$ has a finite nonzero block $A''$ at the top-left, \cite[Theorems 1.7\text{ and }1.8]{bc} obtain the limiting normalized outlier fluctuation when $A''=\theta 1_{\rk(A)}$ and when $A''=J_{\theta,\rk(A)}$ under the hypothesis of bounded fourth moments.

In the case when $A''=vu^*$ is of rank $1$ and is delocalized ($\|u\|_{\infty},\|v\|_{\infty}=O(n^{-1/2})$), they show that the outliers exhibit macroscopic fluctuations and demonstrate a convergence of these fluctuations to the zeros of a Gaussian analytic function. While this phenomenon does not occur with finite rank perturbations, some techniques of the proof are similar to the ones in our proof.

In the setting of finite rank perturbations of iid matrices, when Theorem~\ref{delocthm} is specialized appropriately, our results coincide with \cite[Theorem 2.9]{rochet} for the Ginibre ensemble and with \cite[Theorems 1.7\text{ and }1.8]{bc} for local perturbations of the specified Jordan types. All other cases however, with $X$ having a non-Gaussian atom distribution and $A$ having general eigenvectors (see Remark \ref{rem-conv2}), including the delocalized cases of (ii) and (iii), do not appear to have been explicitly addressed in the literature.

The main technical result of this paper is Proposition~\ref{myclt} which we prove using the moment method. We require a bounded number of moments in all cases and are able to obtain the limiting fluctuations in a more general setting with a unified approach.

The paper is organized as follows. In Section~\ref{secprop} we prove Proposition~\ref{myclt} which characterizes the joint asymptotic distribution of certain random variables arising from powers of $X_n$ appearing in the Neumann series of $(X_n/\sqrt{n}-\lam)^{-1}$. In Section~\ref{secpaths} we prove Lemma~\ref{Sclt}, which determines the joint limiting distribution of random variables related to a normalized resolvent of $X_n$, namely of the form $\sqrt{n}u^*[(X_n/\sqrt{n}-\lam)^{-1}+\lam^{-1}]v$. Using Lemma~\ref{Sclt}, Theorem~\ref{delocthm} is proven in Section~\ref{secthm}, with the help of Proposition~\ref{detpert} from Appendix~\ref{secdet}, a deterministic perturbation result needed to understand the effect of Jordan blocks in perturbations. Appendix~\ref{sec4eps} presents the truncation argument that allows us to assume stronger hypotheses in Proposition~\ref{myclt} and Lemma~\ref{Sclt}.

\subsection{Acknowledgments}
I am indebted to my advisor, Terence Tao, for his constant guidance, support and feedback throughout the course of this work.

\subsection{Notation} In this paper, $n$ will be a parameter going to infinity and many quantities will be implicitly understood to depend on $n$. We will use the asymptotic notation $X=O(Y)$ and $X\ll Y$ to mean there is a constant $C$ independent of $n$, but possibly dependent on other parameters, such that $X\leq CY$ for sufficiently large $n$. Similarly, we write $X=\Omega(Y)$ to mean for some $C$ and sufficiently large $n$, $X\geq CY$. We write $X=o(Y)$ to mean $\lim_{n\to\infty} X/Y\to 0$. For a sequence of events $E=E_n$, we say $E$ occurs with high probability (w.h.p.) if $\P(E_n)=1-o(1)$ and with overwhelming probability if $1-\P(E_n)\ll n^{-c}$ for all $c>0$. We will use $\Rightarrow$ to denote convergence in distribution (and occasionally to denote implication) and finally, we write $[k]$ for $\{1,2,\ldots,k\}$.

\section{A central limit theorem~\ref{myclt}}\label{secprop}

To obtain the limiting fluctuations of the outliers in Theorem ~\ref{delocthm}, we will have to derive the joint asymptotic distributions for certain bilinear averages of the recentered and normalized resolvent, namely for 
\eql{\label{slamdef}S_{\lam}^{u,v}:=-\lam\sqrt{n}u^*((X/\sqrt{n}-\lam)^{-1}+\lam^{-1})v,} with $u$ and $v$ ranging over the generalized eigenvectors of the perturbation matrix $A$. To this end, in this section we prove Proposition~\ref{myclt} which obtains the limiting joint distribution for a bounded number of terms of the Neumann series of ~\eqref{slamdef}. In Lemma~\ref{Sclt} we will control the tail of ~\eqref{slamdef}, thus obtaining its limiting distribution.

Recall the notation introduced in ~\eqref{conjdef} which we reproduce here for convenience. For any complex vector $z$, we let \eq{z^{(d)}:=\left\{
     \begin{array}{lr}
       z & : d=0\\
       \conj{z} & : d=1
     \end{array}
   \right..} 
For $S\subset[n]\times [n]$, we define $X_S=(X^S_{ij})$ through
\eq{X^S_{ij}=\delta_{(i,j)\in S}x_{ij}.}
\begin{prop}\label{myclt} Let $X$ be an iid matrix and $(u_i,v_i)_{i=1}^p=(u^{(n)}_i,v^{(n)}_i)_{i=1}^p$ be a sequence of vectors in $\C^n$. We assume the hypotheses of Theorem~\ref{delocthm} with $(u_i,v_i)_{i=1}^p$ in the place of $(u_{s_r},v_{t_r})_{r\in I_2}$. Thus, in the place of ~\eqref{lim1} and ~\eqref{lim2}, we assume the following limits and define the scalars

\eql{\label{lim1new}C^{(d_1),(d_2)}_{i_1,i_2}:=\lim_{n\to\infty}(u_{i_1})^{(d_1)*}\conj{(u_{i_2})}^{(d_2)}(v_{i_1})^{(d_1)T}(v_{i_2})^{(d_2)}.}

We will assume $\E|x|^m<\infty$ with $m$ defined via ~\eqref{cdef} and ~\eqref{momhyp}.

Define \eq{Z_{i,j}=Z_{i,j}^{(n)}:=\sqrt{n}u^*_i\left(\frac{X_n}{\sqrt{n}}\right)^jv_i} where we have suppressed the $n$ dependence for $X_n$, $u^{(n)}_i$ and $v^{(n)}_i$. Also, for 
\eq{L:=\bigcup_{i\in[p]}\{(k,l)\in [n]\times [n]:|u_{i,k}v_{i,l}|\geq n^{-1/4+\delta}\}} and $L^c:=([n]\times [n])\backslash L$, define 
\eq{Z^L_{i,j}:=\sqrt{n}u^*_i\left(\frac{X_L}{\sqrt{n}}\right)^jv_i} and
\eq{Z^{L^c}_{i,j}:=\sqrt{n}u^*_i\left(\frac{X_{L^c}}{\sqrt{n}}\right)^jv_i.}

For $j=1$, we will assume that the following joint convergences in distribution and define the independent families $(G^L_{i,1})_{i=1}^p$ and $(G^{L^c}_{i,1})_{i=1}^p$ through
\eq{(Z^L_{i,1})_{i=1}^p\Rightarrow (G^L_{i,1})_{i=1}^p} and
\eq{(Z^{L^c}_{i,1})_{i=1}^p\Rightarrow (G^{L^c}_{i,1})_{i=1}^p.} 
Also define $G_{i,1}:=G^L_{i,1}+G^{L^c}_{i,1}$ so that
\eql{\label{G1def}(Z_{i,1})_{i=1}^p\Rightarrow (G_{i,1})_{i=1}^p.} Then for any fixed $m\geq 1$, the $pm$ random variables $(Z_{i,j})_{i=1,j=1}^{p,m}$ converge jointly in distribution to the law of random variables $(G_{i,j})_{i=1,j=1}^{p,m}$ with $(G_{i,j})_{i=1,j=2}^{p,m}$ specified by
\begin{enumerate}[label=(\roman*)]
\item The $G_{i,j}$'s are centered complex Gaussians for $j\geq 2$ with mixed second moments given by 
\eql{\E G^{(d_1)}_{i_1,j}G^{(d_2)}_{i_2,k}=\delta_{jk}(\E x^{(d_1)}x^{(d_2)})^jC^{(d_1),(d_2)}_{i_1,i_2}\label{covariance}. }
\item The collections of random variables $(G_{i,1})_{i=1}^p$ and $(G_{i,j})_{i=1,j=2}^{p,m}$ are independent.
\end{enumerate}

Note in particular that for $j\neq k$, $Z_{i_1,j}$ and $Z_{i_2,k}$ are asymptotically independent.

\end{prop}
\begin{rem}
We note that the case $p=1$ and $c=1$ is a generalization of \cite[Section $4$]{tao} to the complex case with weaker moment assumptions, and is a special case of \cite[Theorems 6.3, 6.4]{bc}. 
\end{rem}
\begin{rem} \label{rem-conv}
The assumption of the joint convergence of $Z^L_{i,1}$ and $Z^{L^c}_{i,1}$ is satisfied under various conditions. We describe some of these below.
\begin{enumerate}[label=(\roman*)]
\item If each $u_i$ and $v_i$ have finite support in $[C]$ independent of $n$, we have the case of a local perturbation and the $G_{i,1}$'s are finite linear combinations of the $x_{i,j}$'s.

\item If each $u_i$ and $v_i$ is uniformly delocalized in the sense that  $\|u_i\|_{\infty}= o(1)$ and $\|v_i\|_{\infty} = o(1)$ for $i\in[p]$, then by the classical central limit theorem, the $G_{i,1}$'s are joint centered complex Gaussians with mixed second moments given by
\eq{\E G^{(d_1)}_{i_1,1}G^{(d_2)}_{i_2,1}=\E x^{(d_1)}x^{(d_2)}C^{(d_1),(d_2)}_{i_1,i_2}. }

\item Each $u_i$ and $v_j$ can be allowed to have a local and a uniformly delocalized part. Namely, we suppose that for some $C$ independent of $n$ and all $i\in[p]$, $\sup_{i>C}|u_i|,\sup_{i>C}|v_i|=o(1)$. In this case, the $G_{i,1}$'s are a sum of a finite linear combination of the $x_{i,j}$'s and an independent Gaussian.

\item Finally, we mention an example that is not contained in the above cases. Let $p=1$, fix $0<r<1$ and set $u_{1,k}=v_{1,k}=r^{k}c_n$ with $c_n$ chosen such that $u^*v=\theta:=2$ say. Then $G_{1,1}$ is an infinite linear combination of the $x_{i,j}$'s with exponentially decreasing entries. \comment{This is similar to a Bernoulli convolution and can give rise to fractal type distributions for suitably chosen $r$.}
\end{enumerate}    
\end{rem}
\subsection{Proof of Proposition ~\ref{myclt}}
Instead of assuming \eqref{momhyp}, via a truncation argument presented in Appendix ~\ref{sec4eps}, it suffices to prove Proposition~\ref{myclt} under the stronger assumption that the atom distribution $x$ satisfies the bound
$|x|\leq K:=o(n^{M})$ with $M=2/m$ given by
\eql{\label{mdef}M=\max(\min(c,1/2),1/4)-\epsilon,} with $c$ defined by ~\eqref{cdef}. Furthermore, by decreasing $c$ slightly (and decreasing $\eps$), we may assume
\eq{\label{cdefnew}\max_{i\in[p]}\|u_i\|_{\infty}\|v_i\|_{\infty}\ll n^{-c}} instead. We will also assume without loss of generality that $(u_i,v_i)_{i\in[p]}$ are unit vectors. 

In step $1$, we show that $(Z^L_{i,1})_{i=1}^p$ is asymptotically independent of 
\eq{(Z^{L^c}_{i,1})_{i=1}^p\cup (Z_{i,j})_{i=1,j=2}^{p,m}.} In step $2$, we derive the joint asymptotic distribution of $(Z^{L^c}_{i,1})_{i=1}^p\cup (Z_{i,j})_{i=1,j=2}^{p,m}$. A key part of the proof is contained in Lemma ~\ref{mainlem}, whose proof we postpone to the end of this section.

Step $2$ employs the moment method which, together with the truncation method (see Appendix~\ref{sec4eps}), contributes to the moment hypothesis. The moment hypothesis decays when the random variables $(Z^L_{i,1})_{i=1}^p$ are dealt with using the moment method; thus we deal with them separately. 

We will need
\begin{lem}\label{lemprob}
Let $A^{(n)}=(A^{(n)}_1,\ldots,A^{(n)}_k)$, $B^{(n)}=(B^{(n)}_1,\ldots,B^{(n)}_k)$ and $C^{(n)}=(C^{(n)}_1,\ldots,C^{(n)}_l)$ be sequences of complex vector valued random variables such that
\eq{(A^{(n)},C^{(n)})\Rightarrow(A,C)\text{ and }B^{(n)}\rightarrow_P 0.}
Then $(A^{(n)}+B^{(n)},C^{(n)})\Rightarrow(A,C)$. In particular, if $A^{(n)}$ and $C^{(n)}$ are independent, then $A^{(n)}+B^{(n)}$ and $C^{(n)}$ are asymptotically independent.
\end{lem}
\begin{proof}
This follows from the Cram\'{e}r-Wold device (see \cite[Chapter 1.7]{billingsley2}) and appears in \cite[Exercise 1.4.2]{billingsley2}. 
\end{proof}

\subsubsection{Step 1} For $j\geq 2$, define 
\eq{Z'_{i,j}:=n^{-(j-1)/2}u_i^*(X-X_{L})^jv_i.} Note that $Z^L_{i_1,j_1}$ and $Z'_{i_2,j_2}$ are functions of disjoint subsets of $\{x_{rs}:r,s\in[n]\}$ and hence, $(Z_{i,j})_{(i,j)\in D}$ and $(Z'_{i,j})_{(i,j)\in D^c}$ are independent. $Z^L_{i_1,j_1}$ and $Z^{L^c}_{i_2,j_2}$ are independent for the same reason.
 
By Lemma~\ref{lemprob}, it suffices to show that
\eql{\label{eqk1}E=E_{i,j}:=n^{-(j-1)/2}u^*_i(X^j-(X-X_L)^j)v_i\rightarrow_P 0} for $i\in [p]$ and $2\leq j\leq m$. We will need the following result.
\begin{lem}\label{lemtao}
Let $u$ and $v$ be unit vectors in $\C^n$ and $X$ be an iid random matrix with atom distribution having mean $0$, variance $1$ and bounded fourth moment. Then
\eql{\E\left|u^*\left(\frac{1}{\sqrt{n}}X\right)^{k}v\right|^2=O\left(\frac{1}{n}\right)} for any fixed $k\geq 1$.
\end{lem}
\begin{rem}
Lemma ~\ref{lemtao} is a special case of Lemma~\ref{pathlemma} which establishes the same statement for $k$ that is allowed to grow polynomially with $n$. We postpone the proof to Subsection \ref{subsec-pl}, where the result is needed in full generality. We remark that Lemma~\ref{lemtao} can also be found in \cite[Lemma 2.3]{tao}.  
\end{rem}
Fix $j\geq 2$ and let $\delta_n=\log n$ (any slowly growing function of $n$ will suffice). 
By Lemma~\ref{lemtao} and Markov's inequality, for any $k\geq 1$,
\eql{\label{taolem}\bigcap_{m=1}^{M}\left\{u_m^*\left(\frac{1}{\sqrt{n}}X\right)^{k}v_m\leq\frac{\delta_n}{\sqrt{n}}\right\}} occurs with high probability for any finite set of $2M$  unit vectors $(u_m)_{m=1}^M$ and $(v_m)_{m=1}^M$.

Recall that
\eq{L:=\bigcup_{i\in[p]}\{(k,l)\in [n]\times [n]:|u_{i,k}v_{i,l}|\geq n^{-1/4+\delta}\}} where $\delta>0$ is fixed. Since $|u_i|_2=|v_i|_2=1$, we have $|L|\ll n^{1/2-2\delta}$. To control, $\|X_L\|$, we will need
\begin{lem}\label{lemnet}
Suppose $S\subset A\times B$ with $\max(|A|,|B|)\leq m$. Then $\|X_S\|\leq O(\log n\sqrt{m})$ w.h.p.
\end{lem}
\begin{proof}
Since $\|X_S\|$ is unchanged when restricting $X_S$ to an $m\times m$ submatrix containing $S$, we may assume $m=n$. If $S=\emptyset$, Lemma~\ref{lemnet} is a consequence of Theorem~\ref{opnorm}. Writing $X':=X_{L}-X_{L^c}$, we have 
\eq{\|X_L\|\leq\frac{1}{2}(\|X\|+\|X'\|)} from the triangle inequality. If the atom distribution $x$ is symmetric, applying Theorem~\ref{opnorm} to $X$ and $X'$ yields the desired bound. To prove the lemma for general $x$, we will need a symmetrization argument from \cite[Section 2.3.2]{tao-book} that we reproduce here for convenience. Letting $X''$ be an independent copy of $X'$, we have
\eq{\E[X'-X''|X']=X'.} Since the operator norm is a convex function, we may apply Jensen's inequality to get
\eq{\|X'\|\leq\E[\|X'-X''\||X'].} Removing the conditioning on $X'$, we have
\eq{\E\|X'\|\leq\E\|X'-X''\|.} Now $X'-X''$ has iid entries, so applying Theorem~\ref{opnorm}, we have 
\al{\P[\|X'\|\geq\log n\sqrt{n}]&\leq \frac{\E\|X'\|}{\log n\sqrt{n}}\\
&\leq \frac{\E\|X'-X''\|}{\log n\sqrt{n}}\\
&=o(1).}
\end{proof}
Applying Lemma~\ref{lemnet} with $m=n^{1/2-2\delta}$ gives
\eql{\label{xlbound}\|X_L\|\ll (\log n)n^{1/4-\delta}\text{ w.h.p.}} 

Now let 
\eq{X^a:= \left\{
     \begin{array}{lr}
       X & : a=0\\
       X_{L} & : a=1
     \end{array}
   \right.} 
	 Expanding ~\eqref{eqk1}, we have
\al{|E|&\leq \sum_{a=1}^j\sum_{\substack{a_1,\cdots,a_j\in\{0,1\}\\\sum a_i=a}}n^{-(j-1)/2}\left|u_i^*X^{a_1}\ldots X^{a_j}v_i\right|\\
&=:\sum_{a=1}^kE_a.}

For $a\geq 2$, 
\al{E_a&\ll\binom{j}{a}\left\|\frac{X}{\sqrt{n}}\right\|^{j-a}\left\|\frac{X_{L}}{\sqrt{n}}\right\|^{a-1}\|X_L\|\\
&=o(1) \text{ w.h.p.},} where we have used ~\eqref{xlbound} and that $a\geq 2$.  

To bound $E_1$, we have 
\al{E_1&\leq\sum_{m=0}^{j-1}\left|u_i^*\left(\frac{X}{\sqrt{n}}\right)^mX_{L}\left(\frac{X}{\sqrt{n}}\right)^{j-1-m}v_i\right|\\
&\leq\sum_{(k,l)\in L}\sum_{m=0}^{j-1}|x_{kl}|\left|u_i^*\left(\frac{X}{\sqrt{n}}\right)^me_k\right|\left|e_l^T\left(\frac{X}{\sqrt{n}}\right)^{j-1-m}v_i\right|\\
&\ll \frac{\delta_n}{\sqrt{n}}\sum_{(k,l)\in L}|x_{kl}|\text{ w.h.p.}} Note that if $j\geq 2$, then either $m\geq 1$ or $j-1-m\geq 1$ for $0\leq m\leq k-1$. Hence the last line follows from ~\eqref{taolem}.
 
Since $\E|x_{kl}|\leq 1$, $\delta_n=\log n$ and $|L|=O(n^{1/2-2\delta})$, we have $E_1\rightarrow_P 0$ by Markov's inequality, and ~\eqref{eqk1} follows.

\subsubsection{Step 2}

We first state and prove the complex version of Wick's theorem (also known as Isserlis' theorem, see \cite{wick}) which will be needed later.

\begin{lem}\emph{(Complex Wick's theorem)}\\
Let $(Z_1,Z_2,\ldots, Z_n) = (X_1+iY_1,\ldots,X_n+iY_n)$ be a centered complex Gaussian vector. Thus the vector $(X_1,Y_1,\ldots,X_n,Y_n)$ is multivariate normal. Then for any $I=(i_1,\ldots,i_{2k})\in[n]^{2k}$, 
\eq{E\prod_{l=1}^{2k}Z_{i_l}=\sum_{P}\prod_{j=1}^k\E [Z_{i_{p_{2j-1}}}Z_{i_{p_{2j}}}]} where the sum is over all partitions $P=\bigcup_{j=1}^k\{p_{2j-1},p_{2j}\}$ of $[2k]$ into pairs. Also, the left hand side is $0$ if $I$ has odd length.
\end{lem}
\begin{proof}
Wick's theorem is the statement of the lemma for multivariate centered real Gaussians. The complex version follows by expanding both sides of the equation into real and imaginary parts and applying Wick's theorem. Let \eq{W_{i}^a= \left\{
     \begin{array}{lr}
       X_i & : a=1\\
       iY_i & : a=2
     \end{array}
   \right.}
	Then \eq{\E\prod_{l=1}^{2k}Z_{i_l}=\sum_{a_1,\ldots,a_{2k}\in\{1,2\}}\prod_{l=1}^{2k}W_{i_l}^{a_l}} while \eq{
	\sum_{P}\prod_{i=1}^k\E [Z_{p_{2j-1}}Z_{p_{2j}}]=\sum_{P}\sum_{a_1,\ldots,a_{2k}\in\{1,2\}}\prod_{i=1}^k\E [W_{p_{2j-1}}^{a_{2j-1}}W_{p_{2j}}^{a_{2j}}].} Switching the sums and applying Wick's theorem to $\E\prod_{l=1}^{2k}W_{i_l}^{a_l}$ for each choice of the $a_l$'s yields the result. 

\end{proof}

We now prove Proposition ~\ref{myclt} for the collection of random variables $(Z^{L^c}_{i,1})_{i=1}^p\cup (Z_{i,j})_{i=1,j=2}^{p,m}$. This part of the proof employs the moment method in a similar way to those in \cite{tao} and \cite{bc}. To avoid notational clutter on a first reading, one may set $p=1$ to grasp the main ideas of the proof. 

To handle the $j=1$ case uniformly, in the proof we will abuse notation by writing $Z_{i,1}$ for $Z^{L^c}_{i,1}$ and $G_{i,1}$ for $G^{L^c}_{i,1}$. When $j=1$, we will denote $X_{L^c}$ by $X^j$ and finally, we define
\eql{\label{uj1def}C^{(d_1),(d_2)}_{i_1,i_2}(j):=\begin{cases}
C^{(d_1),(d_2)}_{i_1,i_2}:j\geq 2\\
\lim_{n\to\infty}\sum_{(k,l)\in L^c}\conj{(u_{i_1,k})}^{(d_1)}\conj{(u_{i_2,k})}^{(d_2)}(v_{i_1,l})^{(d_1)}(v_{i_2,l})^{(d_2)}.
\end{cases}}

By Carleman's theorem for the case of a complex vector of random variables (see e.g. \cite{bai-silver}), it suffices to show that the multivariate mixed moments converge. Namely,
\eql{\label{moment}\E\prod_{\substack{1\leq i\leq p\\1\leq j\leq m}} Z_{i,j}^{r_{i,j}}\overline {Z_{i,j}}^{s_{i,j}}= \E\prod_{\substack{1\leq i\leq p\\1\leq j\leq m}} G_{i,j}^{r_{i,j}}\overline {G_{i,j}^{s_{i,j}}} +o(1)} for $(r_{i,j})_{i=1,j=1}^{p,m}, (s_{i,j})_{i=1,j=1}^{p,m}\in\mathbb{N}^{pm}$.

Let $Q_1:=-\frac{1}{2}\sum_{i,j}(j-1)(r_{i,j}+s_{i,j})$. Then the left hand side of \eqref{moment} is \eql{\label{momprod}
n^{-Q_1}\E\prod_{\substack{1\leq i\leq p\\1\leq j\leq m}} (u_i^*X^jv_i)^{r_{i,j}}(u_i^T\conj{X}^j\conj{v_i})^{s_{i,j}}
.}

Expanding the product in ~\eqref{momprod} will yield terms corresponding to the union of directed paths on the vertex set $[n]$ with $\sum_ir_{i,j}+s_{i,j}$ of them having length $j$ for each $1\leq j\leq m$. We first introduce notation in order to write \eqref{momprod} as a sum $n^{-Q_1}\sum_{*}W(F)$, with $*$ and $W(F)$ defined appropriately.
Next, we reduce the sum to terms with paths having multiplicity two and disjoint interior vertices (see Lemma~\ref{mainlem}). Finally we apply the complex Wick theorem to obtain the proposition. 

Let \eq{S:=\{(a,b,c,d):a\in [p],b\in [m],d\in\{0,1\},c\in [r_{a,b}]\text{ if }d=0 \text{ and } c\in [s_{a,b}] \text{ if } d=1\}} be the index set for the $Z_{i,j}$'s. For $s\in S$ we write $s=(s_a,s_b,s_c,s_d)$. Recalling ~\eqref{conjdef}, \eqref{momprod} can be written as 
\eql{\label{momprod2}n^{-Q_1}\E\prod_{s\in S}(u_{s_a}^*X^{s_b}v_{s_a})^{(s_d)}.} We let
\eq{T:=\{(s,e):s\in S\text{ and }e\in [s_{b+1}]\}} be the index set of terms within the $Z_{i,j}$'s. For $t\in T$, we write \eq{t=(t_s,t_e)=(t_a,t_b,t_c,t_d,t_e).} By a slight abuse of notation, we will write $u_{t}$ for $u_{t_a}$ and $u_s$ for $u_{s_a}$. We denote the index set for terms in the expansion of \eqref{momprod2} by  
\eq{\mathcal{F'}:=\{F:T\rightarrow [n]:t_b=1\Rightarrow (F(t,1),F(t,2))\in L^c\}.} Finally for $s\in S$ and $F\in \mathcal{F'}$ let 
\all{\label{momweight1}W_s(F)&:=(u^{*}_{s,F(s,1)}v_{s,F(s,s_b+1)}\mathbbm{1}_{[s_b\geq 2\text{ or }(F(s,1),F(s,2))\in L^c]})^{(s_d)}(\E\prod_{e=1}^{s_b}x_{F(s,e),F(s,e+1)})^{(s_d)}\\ &=:W_{s,(u,v)}(F)W_{s,x}(F)} 
and set 
\eql{\label{wuvdef}W_{u,v}(F):=\prod_{s\in S} W_{s,(u,v)}(F),}
\eq{W_x(F):=\prod_{s\in S} W_{s,x}(F)} and 
\eql{\label{momweight2}W(F):=\prod_{s\in S}W_s(F).}
 Now we can write \eqref{momprod2} as
\eql{\label{momsum}n^{-Q_1}\E\prod_{s\in S}(u_{s}^*X^{s_b}v_{s})^{(s_d)}=n^{-Q_1}\sum_{F\in\mathcal{F'}}W(F).}

For each partition $\CT=\{T_1,\ldots,T_q\}$ of $T$, set
\eq{\F_{\CT}:=\{F\in\CF:\{F^{-1}(i):i\in[n],F^{-1}(i)\neq\emptyset\}=\{T_1\ldots,T_q\}\}} to be the set of terms $F$ whose preimages induce the partition $\{T_1.\ldots,T_q\}$. We can now write
\eq{n^{-Q_1}\sum_{F\in\CF}W(F)=%\hspace{-0.2cm}
n^{-Q_1}\hspace{-0.5cm}
\sum_{\CT=\{T_1,\ldots,T_q\}}%\hspace{-0.5cm}
%\hspace{-0.2cm}
\sum_{F\in\F_{\CT}}%\hspace{-0.5cm}
W(F).}

We now define notation for the edges of the graph induced by the terms $F$. First, let
$E:=\{(t,t')\in T^2:t_s=t'_s,t'_e=t_e+1\}$ and fix a partition $\CT=\{T_1\ldots,T_q\}$ of $T$. For $F\in \CF_{\CT}$ and $i,j\in [q]=[q(\CT)]$, let 
\eq{E^{\CT}_{i,j}:=\{e=(t,t')\in E:t\in T_i\text{ and }t'\in T_j\}} and let
\eq{E_{\CT}:=\{E^{\CT}_{i,j}:|E^{\CT}_{i,j}|>0\}.}
 Note that $(|e|)_{e\in E_{\CT}}$ is independent of $F\in \CF_{\CT}$ and that 
\eql{\label{wxprod}W_x(F)=\prod_{e\in E_{\CT}}\E|x|^{|e|}.} Since $\E|x|=0$, $W_x(F)=0$ if $|e|=0$ for any $e\in E_{\CT}$. Thus defining
\eq{\CF:=\bigcup_{\substack{\CT\text{ partition of }T:\\|e|\geq 2 \forall e\in E_{\CT}}}\CF_{\CT},}
we have
\eql{\label{momsum2}n^{-Q_1}\sum_{F\in\CF'}W(F)=n^{-Q_1}\sum_{F\in\CF}W(F).}

Each $F\in\CF$ can be interpreted as a union of paths on $[n]$. More precisely, letting $T_s:=\{t\in T:t_s=s\}$, we define $\pi_{F,s}:=F|_{T_s}$ to be the path of $F$ corresponding to term $s\in S$. The interior vertices of $\pi_{F,s}$ are defined to be $F(\{(s,e):e=2,3,\ldots,s_b\})$.

\begin{lem}\label{mainlem}Assume the hypotheses of Proposition~\ref{myclt} and recall the notation introduced above. Let $\CF_0$ be the set of terms $F$ such that each path $\pi_{F,s}$ for ${s\in S}$ has multiplicity $2$ and different paths have disjoint interior vertices. 

Then \eq{n^{-Q_1}\sum_{F\in\CF}W(F)=n^{-Q_1}\sum_{F\in\CF_0}W(F)+o(1).}
\end{lem}
We will postpone the proof of the lemma to the end of the section. Assuming the lemma, we now prove the proposition.

First suppose $\sum_i r_{i,j}+s_{i,j}$ is odd for some $j$. Then $\CF_0$ is empty and the left-hand side of ~\eqref{moment} is $o(1)$ which matches the right-hand side by the vanishing of odd mixed moments of a centered complex Gaussian. For the rest of the proof, we can thus assume that for each $j$, $\sum_i r_{i,j}+s_{i,j}$ is even. 

We group the terms in $\F_0$ as follows. Let $S_j:=\{s\in S:s_b=j\}$ and define $\mathcal{P}_j$ to be the set of unordered partitions of $S_j$ into parts of size two. Note that by assumption, $|S_j|$ is even for all $j$.

For $F\in\CF_0$, note by ~\eqref{momweight1} and ~\eqref{momweight2} that $W(F)$ does not depend on the interior points $\{f(s,e):s\in S, e=2,\ldots,s_b\}$. There are $\sum_{i,j}(j-1)(r_{i,j}+s_{i,j})$ such points which occur in pairs and can be chosen in $n^{Q_1}$ ways.

For $F\in\CF_0$ and $j\in[m]$, let $P_{F,j}\in\CP_j$ be the partition of $\CP_j$ induced by $F$. Then $F$ satisfies the condition that for each part $\{p,q\}\in P_{F,j}$, $F(p,1)=F(q,1)$ and $F(p,p_b+1)=F(q,q_b+1)$. 

Summing over the choices for interior points and $F$ satisfying the above condition instead of summing over $F\in\CF_0$ incurs an $o(1)$ error and we have

\eql{\label{almostdone}n^{-Q_1}\sum_{F\in\CF_0}W(F)=\prod_{j=1}^m\sum_{P_j\in\mathcal{P}_j}\prod_{\{p,q\}\in P_j}\sum_{\substack{F(p,1)=F(q,1),\\F(p,j+1)=F(q,j+1)\in [n]}}W_p(F)W_q(F)+o(1)} where, recalling \eqref{momweight1},
\eq{W_p(F)W_q(F)=(\E x^{(p_d)}x^{(q_d)})^j\prod_{r\in\{p,q\}}\conj{u}^{(r_d)}_{r,F(r,1)}v^{(r_d)}_{r,F(r,j+1)}\mathbbm{1}_{[j\geq 2\text{ or }(F(r,1),F(r,2))\in L^c]}.}

Finally, using \eqref{lim1new} and \eqref{uj1def}, ~\eqref{almostdone} evaluates to 
\eql{\label{lhsfinal}\prod_{j=1}^m\sum_{P_j\in\mathcal{P}_j}\prod_{\{p,q\}\in P_j}(\E x^{(p_d)}x^{(q_d)})^jC^{(p_d),(q_d)}_{p_a,q_a}(j)+o(1).}

On the other hand, we let $\mathcal{P}$ be the set of partitions of $S$ into pairs and for $s\in S$, we set $G_s:=G^{(s_d)}_{s_a,s_b}$. Note that for $j\neq k$, $\E G_{i_1,j}G_{i_2,k}=0$ and hence $G_{i_1,j}$ and $G_{i_2,k}$ are independent. Applying Wick's theorem to the right hand side of \eqref{moment} gives 
\al{\E\prod_{\substack{1\leq i\leq p\\1\leq j\leq m}} G_{i,j}^{r_{i,j}}\overline {G_{i,j}^{s_{i,j}}}
&=\E\prod_{s\in S}G_s\\
&=\prod_{j=1}^m\E\prod_{s\in S_j}G_s\\
&=\prod_{j=1}^m\sum_{P_j\in\mathcal{P}_j}\prod_{\{p,q\}\in P_j}\E G_pG_q
} where we have used Wick's theorem in the third line. Comparing ~\eqref{lhsfinal} and \eqref{covariance} then concludes the proof of the proposition.

Note the following special cases of Proposition ~\ref{myclt}, where we write $G_{i,1}$ for $G^{L^c}_{i,1}$.
	\begin{enumerate}[label=(\roman*)]

	\item If $\E x^2=0$, condition \eqref{covariance} becomes 
		\eql{\label{sp1}\E G_{i_1,j}\conj{G_{i_2,k}}=\delta_{jk}C^{(0),(1)}(i_1,i_2)} and 
		\eq{\E G_{i_1,j}G_{i_2,k}=0.}

	\item If we further assume that for $p=d^2$, the vectors $(u_i,v_i)_{i=1}^p$ are of the form $(u_a,u_b)_{a,b=1}^d$ with $(u_a)_{a=1}^d$ orthonormal, then  \eqref{sp1} reduces to 
		\eql{\label{sp2}\E G_{(a,b),j}\conj{G_{(c,d),k}}=\delta_{jk}\delta_{ab}\delta_{cd}.}
	\end{enumerate}

\subsection{Proof of Lemma~\ref{mainlem}}

Fix a partition $\CT=\{T_1,\ldots,T_q\}$ of $T$ with $|e|\geq 2$ for every $e\in E_{\CT}$. We first rewrite the sum $n^{-Q_1}\sum_{F\in\F_{\CT}}W(F)$ as a product of terms over $j\in [q]$.

Define $T^1:=\{t\in T:t_e=1\}$, $T^2:=\{t\in T:t_e=t_b+1\}$, $T^3:=T\backslash(T^1\cup T^2)$ and let $T_j^l:=T_j\cap T^l$ for $l=1,2,3$. For $t\in T$ and $i\in [n]$, define the vertex weights
\eql{\label{wtidef}w(t,i):=\begin{cases}
|u_{t,i}| : t\in T^1\\
|v_{t,i}| : t\in T^2\\
n^{-1/2} : t\in T^3
\end{cases}..} The $w(t,i)$'s account for the factors $n^{-Q_1}$ and $W_{u,v}(F)$ in ~\eqref{wuvdef} and ~\eqref{momsum} respectively. Since $\E|x|^a\ll K^{(a-4)_+}$, using ~\eqref{wxprod} we have
\all{\label{eqfmain}\notag n^{-Q_1}\left|\sum_{F\in\CF_{\{T_1,\ldots,T_q\}}}W(F)\right|
&\notag\leq\sum_{\substack{i_1,\ldots,i_q\in[n]\\\text{distinct}}}\left(\prod_{j=1}^q\prod_{t\in T_j}w(t,i_j)\right)
\prod_{e\in E_{\CT}}K^{(|e|-4)_+}\\
&\leq\sum_{i_1,\ldots,i_q\in[n]}\left(\prod_{j=1}^q\prod_{t\in T_j}w(t,i_j)\right)
\prod_{e\in E_{\CT}}K^{(|e|-4)_+}.}
We would like to bound $\prod_{e\in E_{\CT}}K^{(|e|-4)_+}$ by $\prod_{t\in T}K^*(t,i_{j(t)})$ for some suitably defined $K^*$ in order to bound the right-hand side ~\eqref{eqfmain} by
\eq{\prod_{j=1}^q\sum_{i_1,\ldots,i_q\in[n]}\prod_{t\in T_j}w(t,i_j)K^*(t,i_j).}
We do this first for the expression $\prod_{e\in E_{\CT}}K^{|e|}$ in order to motivate some of the technical definitions. 
Fix $i_1,\ldots,i_q\in[n]$ and assume for $t\in T$ and $j\in[q]$ that $|u_{t,i_j}|,|v_{{t},i_j}|\neq 0$. Recall the parameter $c\in [0,1]$ from ~\eqref{cdefnew}. For $t\in T_j$, $t^1\in T_j^1$ and $t^2\in T_j^2$, define 
\eql{\label{ktidef}K(t,i):=\begin{cases}
       |u^{-(1-\eps)}_{t,{i}}| \quad: t\in T^1\\
			 |v^{-(1-\eps)}_{t,i}| \quad: t\in T^2\\
			\max(K^2n^{-c(1-\eps)},K) \quad: t\in T^3
     \end{cases}.} 
We first show that 
\eql{\label{eqkspread}\prod_{e\in E_{\CT}}K^{|e|}\ll\prod_{j=1}^q\prod_{t\in T_j}K(t,i_j).}

Fix $s\in S$. Suppoes $s_b=1$. Then for $\delta$ and $\eps$ sufficiently small, 
\all{\label{kcheck1}\notag\prod_{t\in T:t_s=s}K(t,i_{j(t)})&\geq\min_{(k,l)\in L^c}|u_{t,k}v_{t,l}|^{-(1-\eps)}\\
&\gg n^{\max(1/4-\delta,c)(1-\eps)}\notag\\
&\gg K.} The last line follows from $M<\max(1/4,c)$ which is a consequence of ~\eqref{mdef}, .

If $s_b\geq 2$,
\all{\label{kcheck2}\notag\prod_{t\in T:t_s=s}K(t,i_{j(t)})
&\gg K^2n^{c(1-\eps)}\left(\|u_{t}\|_{\infty}\|v_{t}\|_{\infty}\right)^{-(1-\eps)}K^{s_b-2}\\
&\gg K^{s_b}} where we have used $\|u_t\|_{\infty}\|v_t\|_{\infty}\ll n^{-c}$. Using ~\eqref{kcheck1} and ~\eqref{kcheck2} and taking the product over $s\in S$ gives ~\eqref{eqkspread}. We now define $K^*(t,i)$ in such a way that we have the analogous bound
\eql{\label{k*spread}\prod_{e\in E_{\CT}}K^{(|e|-4)_+}\ll\prod_{j=1}^q\prod_{t\in T_j}K^*(t,i_j).} 

First, order the elements of $T^l_j=\{t^l_{1},t^l_{2},\ldots,t^l_{|T^l_j|}\}$ arbitrarily for $l=1,2,3$. We define the set $C_j\subset T_j$ by the following conditions.
\begin{enumerate}[label=(\roman*)]
\item $t^3_{k}\in C_j\iff k\leq 2$.
\item For $l=1,2$, $t^l_{k}\in C_j\iff k+|T^3_j|\leq 2$.
\end{enumerate} It is easy to verify that $|C_j\backslash T^1_j|,|C_j\backslash T^2_j|\leq2$.	We now define  
\eq{K^*(t,i):=\begin{cases}
          1:t\in C_j\\
					K(t,i): \text{ otherwise }. 
					\end{cases}.}	
We now prove ~\eqref{k*spread}. Fix $e\in E_{\CT}$ and suppose $e\subset T_i\times T_j$. Define $e'\subset e$ by
\eq{e':=\{(s,t)\in e:s\in C_i\text{ or }t\in C_j\}.} Since $|C_i\backslash T^2_i|,|C_j\backslash T^1_j|\leq2$, $|e'|\leq 4$ and we have
\eq{\prod_{e\in\CT}K^{(|e|-4)_+}\leq \prod_{e\in\CT}K^{|e\backslash e'|}.} It thus suffices to show
\eq{\prod_{e\in\CT}K^{|e\backslash e'|}\leq \prod_{j=1}^q\prod_{t\in T_j}K^*(t,i_j).} As in the proof of ~\eqref{eqkspread}, we fix $s\in S$. Let $C:=\bigcup_{j\in[q]}C_j$ and define
\eq{e_s=\{((s,l),(s,l+1)):1\leq l\leq s_b\text{ and }(s,l),(s,l+1)\notin C\}} and
\eq{v_s=\{(s,l):(s,l)\notin C\text{ and }l=2,3,\ldots,s_b\}.}
Since $K(t,i)\geq 1$ for $t=(s,1)$ and $t=(s,s_b+1)$, it suffices to show 
\eq{K^{|e_s|}\leq\prod_{t\in v_s}K(t,i).} If $|e_s|=s_b$, this follows from ~\eqref{kcheck1} and ~\eqref{kcheck2}. Now suppose $|e_s|<s_b$. We first show that $|e_s|\leq|v_s|$. Choose $l^*$ such that $(s,l^*)\in C$ and define the map 
$f:e_s\rightarrow v_s$ by
\eq{f((s,l),(s,l+1)):=\begin{cases}
(s,l+1):l\leq l^*-2\\
(s,l):l\geq l^*+1
\end{cases}.} We see that $f$ is injective and hence $|e_s|\leq |v_s|$. Since $K(t,i)\geq K$ for $t\in v_s$, we have
\al{K^{|e_s|}&\leq K^{|v_s|}\\
&\leq \prod_{t\in v_s}K(t,i)} completing the proof of ~\eqref{k*spread}.  

We can now use \eqref{k*spread} in \eqref{eqfmain} to write 
\all{\notag\left|n^{-Q_1}\sum_{F\in\CF_{\{T_1,\ldots,T_q\}}}W(F)\right|&\leq\sum_{i_1,\ldots,i_q\in[n]}\prod_{j=1}^q\prod_{t\in T_j}w(t,i_j)K^*(t,i_j)\\
&=\prod_{j=1}^q\left(\sum_{i_j\in[n]}\prod_{t\in T_j}w(t,i_j)K^*(t,i_j)\right)\\
&=:\prod_{j=1}^qW^*(T_j).} 
We now fix a part of $\CT$, say $T_1$ and consider $W^*(T_1)$. To prove Lemma ~\ref{mainlem}, it suffices to prove the following.

\begin{lem}\label{lem1}
\begin{enumerate}[label=(\roman*)]
\item\label{item1} $W^*(T_1)=O(1)$
\item\label{item2} If $|T_1^3|\geq 1$, then $|W^*(T_1)|=o(1)$ unless $|T^3_1|=|T_1|=2$.
\item\label{item3} $\prod_jW^*(T_j)=o(1)$ unless $|e|=2$ for every $e\in E_{\CT}$.
\end{enumerate}
\end{lem}
\begin{proof}
We first show that 
\eql{w(t,i)K(t,i)=\begin{cases}
O(1):t\in T_1^1\cup T_1^2\\
o(1):t\in T^3_1\label{t3bound}
\end{cases}}
using ~\eqref{wtidef}, ~\eqref{ktidef} and ~\eqref{mdef}. Suppose $t\in T^1$. Then $w(t,i)K(t,i)\leq |u_i|^{\eps}=O(1)$. We have a similar bound for $t\in T^2$. Finally, if $t\in T^3$, then 
\eq{w(t,i)K(t,i)=n^{-1/2}\max(K^2n^{-c(1-\eps)},K).} Since $K=o(n^M)$ and $M\leq \min(1/2, c)$, we have the desired bound. This implies in particular that for any $D\subset C_1$, 
\eql{\label{dsum}W^*(T_1)\ll\sum_{i\in[n]}\prod_{t\in D}w(t,i).} 

We prove Lemma~\ref{lem1}.\ref{item2} first. For $u$ and $v$ unit vectors in $\C^n$, we will need the estimate 
\eql{\label{uepsbound}\sum_{i\in[n]}|u_i|^{\eps}\ll O(n^{1-\eps/2})} which follows from H\"{o}lder's inequality. 
Suppose $|T^3_1|=1$. Then, since each edge has multiplicity at least $2$, we must have $|T^1_1|,|T^2_1|\geq 1$. Applying ~\eqref{dsum} with $D=C_1=\{t^1_1,t^2_1,t^3_1\}$, we have that for some $(u,v)$, 
\al{W^*(T_1)&\ll \sum_{i=1}^n n^{-1/2}|u_i||v_i|\\
&\leq O(n^{-1/2}).} If $|T^3_1|\geq 3$, then $C_1=\{t^3_1,t^3_2\}$ and
\al{W^*(T_1)&\ll \sum_{i\in[n]}n^{-1}K^*(t^3_3,i)\\
&=o(1)} by ~\eqref{t3bound}. Finally, suppose $|T^3_1|=2$. Suppose $|T^1_1|\geq 1$. Then from ~\eqref{uepsbound}, we have 
\eq{W^*(T_1)\ll\sum_{i\in[n]}n^{-1}|u_i|^{\eps}=o(1).} We have a similar estimate if $|T^2_1|\geq 1$. We conclude that if $|T^3_1|\geq 1$, $W^*(T_1)=o(1)$ unless $|T^3_1|=2$ and $|T^1_1|=|T^2_1|=0$, in which case $W^*(T_1)=O(1)$.

We now prove \ref{item3}. Assume first that $e$ is an edge incident to distinct vertices, say $e\subset T_1\times T_2$, and that $|e|\geq 3$. By \ref{item2}, we may assume $T^3_1=T^3_2=\emptyset$. Since $|T^1_1|,|T^2_2|\geq 3$, we may choose $(s_i,t_i)\in e$ for $i=1,2,3$ where $s_i\in T^1_1$ and $t_i\in T^2_2$ and let $C_1=\{s_1,s_2\}$ and $C_2=\{t_1,t_2\}$. Then bounding $W^*(T_1)W^*(T_2)$ by the contribution from $(s_i,t_i)_{i=1}^3$, we have
\al{W^*(T_1)W^*(T_2)&\ll \sum_{(i,j)\in L^c}\prod_{k=1}^2|u_{s_k,i}v_{t_k,j}||u_{s_3,i}v_{t_3,j}|^{\eps}\\
&\leq \max_{(i,j)\in L^c}|u_{s_3,i}v_{t_3,j}|^{\eps}\sum_{i\in[n]}|u_{s_1,i}||u_{s_2,i}|\sum_{j\in[n]}|v_{t_1,j}||v_{t_2,j}|\\
&=o(1).} 
We have a similar bound if $e$ is a loop at say $T_1$.

To complete the proof of the lemma, it remains to prove \ref{item1} in the cases not covered by \ref{item2} and \ref{item3}. Thus, set $|T_1^3|=0$ and assume without loss of generality that $|T_1^1|\geq 2$. Then with $D=\{t^1_1,t^2_1\}=:\{s,t\}$ in ~\eqref{dsum} we have 
\al{W^*(T_1)&\ll \sum_{i\in[n]}|u_{s,i}u_{t,i}|\\
&=O(1).}
\end{proof}

\section{Proof of Lemma~\ref{Sclt}}\label{secpaths}

Recall the bilinear average of the normalized resolvent introduced in ~\eqref{slamdef} in Section \ref{secprop}. In this section we control the tail of its Neumann series and, with the help of Proposition~\ref{myclt}, obtain the joint limiting distribution of such terms in Lemma~\ref{Sclt}. This is the main ingredient in the proof of Theorem~\ref{delocthm} which is presented in the next section.
 
\begin{lem}\label{Sclt}
Fix complex numbers $\theta_1,\ldots,\theta_a$ with $|\theta_j|>1$ for $j\in[a]$ and suppose $\lambda_j=\lambda_{n,j}\rightarrow_{P}\theta_j$ as $n\rightarrow\infty$. Let $(u_i,v_i)_{i=1}^p$ be $p$ pairs of vectors satisfying the hypotheses of Proposition \ref{myclt}. Let
\eq{S_{i,j}:=\sum_{k\geq 1}\frac{\sqrt{n}\left<(\frac{X}{\sqrt{n}})^kv_i,u_i\right>}{\lam_j^k}=:\sum_{k\geq 1}\frac{Z_{i,k}}{\lam_j^k}.} Recall the definition of $(G_{i,1})_{i=1}^p$ from Proposition ~\ref{myclt} and define centered complex Gaussians $(g_{i,j})_{i=1,j=1}^{p,a}$ independent of $(G_{i,1})_{i=1}^p$ with mixed second moments given by 
\eql{\label{Smom}\E g^{(d_1)}_{i,j}g^{(d_1)}_{i',j'}=
\frac{(\E x^{(d_1)}x^{(d_2)})^2}{\theta_j\theta_{j'}(\theta_j\theta_{j'}-\E x^{(d_1)}x^{(d_2)})}
U^{(d_1),(d_2)}_{i,i'}V^{(d_1),(d_2)}_{i,i'}.} Then
\eq{(S_{i,j})_{i=1,j=1}^{p,a}\Rightarrow (F_{i,j})_{i=1,j=1}^{p,a}} where 
\eql{\label{Fdeflem}F_{i,j}:=\frac{G_{i,1}}{\theta_j}+g_{i,j}.} 
\end{lem}
To prove the lemma, we split $S_{i,j}$ into three sums as follows. Fix cutoffs $m>0$ and $T_n=\log^2n$ ($T_n=\omega(\log n)$ suffices) and define 
\al{S_{i,j}&=
\sum_{k=1}^m\frac{Z_{i,k}}{\lam_j^k}+
\sum_{k=m+1}^{T_n}\frac{Z_{i,k}}{\lam_j^k}+
\sum_{k>T_n}^{\infty}\frac{Z_{i,k}}{\lam_j^k}\\
&=:S_{i,j}^A+S_{i,j}^B+S_{i,j}^C.}
We define \eql{\label{Tdef}T_{i,j}^A:=\sum_{k=1}^m\frac{G_{i,k}}{\theta_j^k}} where the $G_{i,k}$ are defined as in the statement of Proposition~\ref{myclt}. Note that $T_{i,j}^A$ is independent of $n$.

By Proposition \ref{myclt} and the multivariate version of Slutsky's theorem (see \cite{billingsley2}), 
\eq{((Z_{i,k}),(\lam_j))\Rightarrow ((G_{i,k}),(\theta_j)),} where the joint convergence is over all $i\in[p]$, $k\in[m]$ and $j\in[a]$. By the continuous mapping theorem, $(S_{i,j}^A)\Rightarrow(T_{i,j}^A)$ jointly for $i\in[p]$ and $j\in[a]$. By the definitions of $T_{i,j}^A$ in ~\eqref{Tdef} and of $G_{i,k}$ in ~\eqref{G1def} and ~\eqref{covariance}, and by inspecting ~\eqref{Smom} and ~\eqref{Fdeflem}, we see that
\eq{T^A_{i,j}\stackrel{m\rightarrow\infty}{\Longrightarrow} F_{i,j}} jointly.

To prove Lemma~\ref{Sclt}, it suffices to prove

\begin{lem}\label{Svar}
\begin{enumerate}[label=(\alph*)]
\item $\lim_{m\rightarrow\infty}\lim_{n\rightarrow\infty}\E|S^B|=0$ and
\item $\lim_{m\rightarrow\infty}\lim_{n\rightarrow\infty}\E|S^C|=0$.
\end{enumerate}
where we have suppressed the $i$ and $j$ dependence for $S^B_{i,j}$ and $S^C_{i,j}$.
\end{lem}

Define the event 
\eql{\label{event}E_n:=\{|\lam_{n,j}-\theta_j|<\delta_j:=\frac{|\theta_j|-1}{4}\text{ for all }j\in [a]\}.} 
By hypothesis $\P(E_n)=1-o(1)$ so it suffices to prove Lemma~\ref{Sclt} (and hence Lemma~\ref{Svar}) on $E_n$. In the following, we fix an index $j$ and set $\delta:=\frac{|\theta|-1}{4}$. Note that we have
\eql{\label{lamtheta}|\lambda|>1+\frac{3}{4}(|\theta|-1).}
We prove Lemma~\ref{Svar}b first.
\begin{proof}
Recall that on $E_n$, $|\lam|>1+3\delta$ (see ~\eqref{event}). By Theorem~\ref{opnorm}, $\rho(X/\sqrt{n})<1+\delta$ w.h.p. and we can choose $l$ such that $\|(\frac{X}{\sqrt{n}})^l\|^{1/l}<1+2\delta$. We may assume without loss of generality that these events occur on $E_n$.  By submultiplicativity of the operator norm, 
\al{\left\|\left(\frac{X}{\sqrt{n}}\right)^k\right\|&\leq  \left\|\left(\frac{X}{\sqrt{n}}\right)^l\right\|^{\lfloor\frac{k}{l}\rfloor}\max_{0\leq i<l}\left\|\left(\frac{X}{\sqrt{n}}\right)^i\right\|\\
&\leq O_l(1+2\delta)^k \text{w.h.p.}}

By the Cauchy-Schwarz inequality, we have 
\al{|S_C|&\leq \sum_{k> T_n}\frac{\sqrt{n}\left\|\left(\frac{1}{\sqrt{n}}X\right)^k\right\||u|_2|v|_2}{|\lam|^k}\\
&<O_l(\sqrt{n})\sum_{k> T_n}\left(\frac{1+2\delta}{1+3\delta}\right)^{k}\\
&=o(1)} where the last line follows from our choice of $T_n=\log^2n$.
\end{proof}

To prove Lemma~\ref{Svar}a, we will need 
\begin{lem}\label{pathlemma}
Let $u$ and $v$ be unit vectors in $\C^n$ and set
\eq{Z_k:=\sqrt{n}u^*\left(\frac{1}{\sqrt{n}}X\right)^kv.}

Fix $\epsilon>0$ and assume $|x|\leq K=O(n^{\frac{1-\epsilon}{2}})$. Then there exists $c=c(\epsilon)>0$ such that for all $k\ll n^c$, 
\eql{\E |Z_k|^2=O(1)\label{pl2}.} 
\end{lem}

Assuming Lemma ~\ref{pathlemma} we prove Lemma~\ref{Svar}a on $E_n$. Since $(\E|Z|)^2\leq\E|Z|^2$, we have
\al{\E |S^B|&\leq\sum_{k=m+1}^{T_n}\E\frac{|Z_k|}{|\lam|^k}\\
&\ll|1+\frac{3}{4}(|\theta|-1)|^{-2m}}
where we have used Lemma~\ref{pathlemma} and ~\eqref{lamtheta} in the last line. Lemma~\ref{Svar}(a) follows from letting $m\rightarrow\infty$.

\begin{rem}
Note that by the truncation argument given in Appendix~\ref{sec4eps}, Lemma~\ref{pathlemma}, and hence Lemma~\ref{Svar}a, is valid under the moment hypothesis $\E|x|^{4+\epsilon}<\infty$ for any fixed $\epsilon>0$.
\end{rem}

\subsection{Proof of Lemma 9}\label{subsec-pl}
In this subsection we prove Lemma ~\ref{pathlemma}.

\begin{proof} It suffices to show 
\eql{\label{zpower}\E |u^*X^kv|^2=O(n^{k-1}).} 
Let 
\eq{T:=\{(a,b):a=1,2, b=0,1,\ldots,k\},} 
\eq{T':=\{(a,b)\in T:b<k\}} and 
\eq{E:=\{((a,b),(a,b+1))\in T^2:b<k\}.} Let $T_P:=T|_{a=1}$, $T_Q:=T|_{a=2}$ and for $t\in T'$, set $t^s:=(a,b+1)$.  
We will designate the terms in the expansion of ~\eqref{zpower} by
\eq{\mathcal{P'}:=\{F:T\rightarrow [n]\}.} For $F\in\mathcal{P'}$, let $F_P:=F|_{T_P}$ and $F_Q:=F|_{T_Q}$. Let
\eq{W_{u,v}(F):=|u_{F(1,0)}u_{F(2,0)}v_{F(1,k)}v_{F(2,k)}|} and 
\eq{W_x(F):=\E|\prod_{t\in T'}x_{F(t),F(t^s)}|.} Then we have
\eql{\label{prepathsum}\E |u^*X^kv|^2\leq \sum_{F\in \mathcal{P'}}W_{u,v}(F)W_x(F).} For $F\in\mathcal{P'}$, let 
\eq{E_{F}:=\{(F(t),F(t^s))\in[n]^2: t\in T'\}.} denote the edges of $F$ and let 
\eq{\CE_F:=\{\{t\in T':(F(t),F(t^s))=(i,j)\}:(i,j)\in E^F\}.} Then 
\eq{W_x(F)=\prod_{e\in \CE^F}\E|x|^{|e|}.} Noting that $\E|x|=0$ and letting
\eq{\CP:=\{F\in\CP':|e|\geq 2\text{ for all }e\in \CE^F\},} we have 
\eql{\label{pathsum} \E|u^*X^kv|^2\leq\sum_{F\in\CP}W_{u,v}(F)W_x(F).}

Now, for a fixed $F\in\CP$, let 
\eq{V=V_F:=\{F(t):t\in T\}} be the set of vertices. For $v\in V$ let $m(v)=|F^{-1}(v)|$ denote its multiplicity. Let $d_{\text{in}}(v):=|\{x\in[n]:(x,v)\in E\}|$ and $d_{\text{out}}(v):=|\{x:(v,x)\in E\}|$ denote its indegree and outdegree. Finally, let $d(v):=d_{\text{in}}(v)+d_{\text{out}}(v)$ be the (total) degree of $v$. 

Shown in Figure~\ref{egpath} is an example with $k=4$ with the paths $(1,2,3,4)$ and $(2,3,4,1)$. Each vertex has indegree $2$ and outdegree $2$. 
\begin{figure}
    \centering
    \includegraphics[scale=1.2]{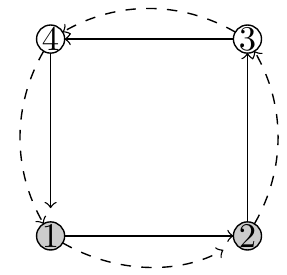}
    \caption{An example of $F\in\CP$ with $k=4$.}
    \label{egpath}
\end{figure}

We will first determine the main term from $\CP$ and its contribution to \eqref{pathsum}. 
\begin{lem}\label{lempmain}
Suppose $F\in\CP$. Then $|V|\leq k+1$ and that equality occurs only when $F_P=F_Q$ and $|F_P|=|F_Q|=k+1$.  
\end{lem}
Fix $v\in V_F$ and suppose $d(v)=1$. Since each edge has multiplicity at least two, we have the following.
\begin{enumerate}[label=(\roman*)]
\item If $d_{\text{in}}(v)=1$, $v=F(1,0)=F(2,0)$.
\item If $d_{\text{out}}(v)=1$, $v=F(1,k)=F(2,k)$
\end{enumerate} In particular, if two vertices of $V$ have degree $1$, then one has outdegree $1$, the other has indegree $1$ and the rest have both outdegree and indegree of at least $1$. Since $|e|\geq 2$ for each $e\in \CE_F$, we also have $|E_F|\leq k$. Thus   
\al{ 2k\geq 2|E_F|&=\sum_{v\in V}d(v)\\ 
&\geq 1+1+2(|V|-2)\\
&=2(|V|-1).}
Thus, $|V|\leq k+1$ with equality occurring only when two of the vertices have degree $1$ and the rest have degree $2$. This proves the lemma.

We let $\CP_{\text{main}}:=\{F\in\CP:|V_F|=k+1\}$. We also let 
\eq{\CP'_0:=\{F\in\CP: |V_F|=k,F_P=F_Q,F(1,0)=F(2,0)=F(1,k)=F(2,k)\}.} Then, the contribution of 
$\CP_{\text{main}}\cup \CP'_0$ to \eqref{pathsum} is given by 
\eq{\sum_{\substack{F(1,0)=F(2,0)\in[n]\\F(1,k)=F(2,k)\in[n]}}|u_{F(1,0)}|^2|u_{F(1,k)}|^2n^{k-1}=n^{k-1}.}

We partition the remainder of $\CP$ in the following way. First let 
\eq{T_1:=\{(1,0),(2,0),(1,k),(2,k)\}\subset T} be the terms corresponding to the starts and ends of the paths. For $t\geq 0$ and $P$ a partition $T_1$ with $|P|\geq 2$ if $t=0$, let 
\eq{\CP_{P,t}:=\{F\in\CP:|V_F|=k-t,F(s)=F(t)\Leftrightarrow s\sim_P t,s,t\in T_1\}.} Note that we exclude the trivial partition $P=\{T_1\}$ when $t=0$ since $\CP_{\{T_1\},0}=\CP'_0$. We let $\CP_0=\bigcup_{P\neq \{T_1\}}\CP_{P,0}$ and for $t>0$, we let $\CP_t=\bigcup_P\CP_{P,t}$.

\begin{lem}
For $F\in\CP_{t}$, $W_x(F)\ll K^{2t}$.
\end{lem}

Since $\E|x|^a\ll K^{(a-4)_+}$, 
\al{W_x(F)&\leq \prod_{e\in\CE_F}\E|x|^{|e|}\\
&\ll \prod_{e\in\CE_F} K^{(|e|-4)_+}.}
It suffices to show that $\sum_{e\in\CE_F}(|e|-4)_+\leq 2t$. Since at most one vertex has no outgoing edge, $|\CE_F|\geq k-t-1$. Also the $|e|$'s satisfy $\sum_{e\in\CE_F}|e|=2k$ and $|e|\geq 2$.  If $|e|\leq 4$ for all $e\in\CE_F$, there is nothing to prove. If $|e_1|\geq 4$ say, then 
\al{\sum_{e\in\CE_F}(|e|-4)_+&=|e_1|-4+\sum_{e\neq e_1}(|e|-4)_+\\
&\leq \sum_{e\in\CE_F}(|e|-2)-2\\
&\leq 2k-2(k-t-1)-2=2t.} 

We now turn to controlling $S_{p,t}:=|\sum_{F\in\CP_{P,t}}W_{u,v}(F)|$. To simplify notation, we will do this for the specific case $P=\{\{(1,0)\},\{(2,0)\},\{(1,k),(2,k)\}\}$. We can bound $S_{P,t}$ by
\eq{\sum_{\substack{i_1,i_2,i_3\in [n]\\\text{distinct}}}
|u_{i_1}u_{i_2}v_{i_3}^2||\{F\in\CP_{t}:F(1,0)=i_1,F(2,0)=i_2,F(1,k)=F(2,k)=i_3\}|.} The cardinality of the last set is independent of the choice of indices $i_1,i_2$ and $i_3$, and in fact only depends on size of the partition $P$. We denote it by $N_{|P|}$. Removing the restriction to distinct indices and using $\sum_i|u_i|=O(\sqrt{n})$, we may bound the contribution as $nN_{|P|}$. 

The case for a general partition is similar and we have the bound 
\eq{S_{P,t}\leq n^{c_P/2}N_{|P|}} where $c_P$ is the number of singletons in the partition $P$. To determine $N_{|P|}$, we first choose the remaining vertices of $V_F$ in $\binom{n}{k-t-|P|}$ ways. We let $N_2=N_2(t)$ be the maximum number of ways to choose $E_F$, over $P$ and $V_F$. Similarly, we let $N_3=N_3(t)$ be the maximum number of ways to choose $\CE_F$, over $P$, $V_F$ and $E_F$. Since
\eq{\binom{n}{k-t-|P|}\leq \frac{n^{k-t-|P|}k^{|P|}}{(k-t)!},} we have
\eq{S_{P,t}\leq (n^{c_P/2-|P|}k^{|P|})\frac{n^{k-t}}{(k-t)!}N_2(t)N_3(t),} with $|P|\geq 2$ if $t=0$. Considering the possibilities for $P$ and setting 
\eq{S_t:=\sum_{P\text{ partition of }T_1}S_{P,t},} we have
\eql{\label{scontrib}S_{t}\ll\left\{
     \begin{array}{ccc}
       k^2n^{k-3/2}N_2N_3/{k!} &:& t=0\\
       kn^{k-t-1}N_2N_3/{(k-t)!}  &:& t\geq 1\\
     \end{array}
   \right..} 

We now estimate $N_2=N_2(t)$, the number of ways to choose the set of edges $E_F$ for $F\in\CP_t$. As observed earlier, at least $k-t-1$ vertices have positive outdegree, and similarly for the indegree. We need to assign at most $k$ oriented edges to the $k-t$ vertices such that these conditions are met. Recall $d_{\text{out}}(i)$ to be the outdegree of vertex $i$. We will allow for repetitions when choosing the edges to include graphs with less than $k$ edges. Hence we may impose the constraint $\sum_{i=1}^{k-t}d_{\text{out}}(i)=k$. For at least $k-t-1$ vertices, $d_{\text{out}}(i)\geq 1$. This gives\footnote{This follows from the standard \textit{stars and bars} combinatorial argument; see \cite{feller}.} $\binom{k}{t+1}$ ways of choosing the outdegrees $(d_{\text{out}}(i))_{i=1}^{k-t}$. To assign the incoming edges of the vertices, we partition the $k$ edges into $k-t$ nonempty parts $(E_i)_{i=1}^{k-t}$. We first choose $k-t$ edges to belong to the different $E_i$'s and then we choose parts for each of the remaining edges. This can be done in at most $\binom{k}{t}(k-t)^t$ ways. Finally, we assign the $k-t$ parts to the vertices with positive indegree. If all $k-t$ vertices have incoming edges, there are at most $(k-t)!$ ways to assign each of them an $E_i$. Now suppose only $k-t-1$ of the vertices have incoming edges. First, there are at most $(k-t)^4$ ways to choose $2$ vertices and $2$ parts, with one vertex being assigned both parts and the other having no incoming edges. Next, there are $(k-t-2)!$ ways of assigning the remaining parts to the remaining vertices. Hence 
\all{\notag N_2&\leq\binom{k}{t+1}\binom{k}{t}(k-t)^t((k-t)!+(k-t)^4(k-t-2)!)\\
\notag&\leq \frac{k^{2t+1}}{(t+1)!t!}k^t(k-t)!k^2\\
\label{n2bound}&\leq\frac{k^{3t+3}(k-t)!}{(t+1)!t!}.}

We now estimate $N_3=N_3(t)$, the number of ways of choosing $\CE_F$ once $V_F$ and $E_F$ have been chosen. Since each vertex has at least one outgoing edge, the maximum outdegree of any vertex is at most $t+1$. On the other hand, since $d_{\text{out}}(1)+\ldots+d_{\text{out}}(k-t)\leq k$, at least $\max(k-2t,0)$ vertices have $d_{\text{out}}(i)=1$. At least $\max(2k-4t,0)$ legs start from these vertices so at most $4t$ legs begin at vertices with $d_{\text{out}}(i)>1$. At each of these legs, we have at most $t+1$ choices to make when choosing the path. We thus have 
\eql{\label{n3bound}N_3\leq (t+1)^{4t},} which is independent of the chosen vertices and edges. 

For $t=0$, using ~\eqref{scontrib}, ~\eqref{n2bound} and ~\eqref{n3bound}, we have
\al{S_0&\leq n^{k-3/2}k^2/{k!}N_2(0)N_3(0)\\
&\leq k^5n^{k-3/2}.} Since $W_x(F)=O(1)$ for $t=0$, the contribution to ~\eqref{pathsum} is $o(n^{k-1})$.

For $t\geq 1$, we have
\al{S_tK^{2t}&\ll n^{k-t-1}kN_2N_3K^{2t}/{(k-t)!}\\
&\leq k^4n^{k-1}\frac{k^{3t}(t+1)^{4t}}{n^{\eps t}(t+1)!t!}\\
&\leq k^4n^{k-1}\frac{k^{3t}(te)^{2t}}{n^{\eps t}},} where we have used the estimates $t!>\frac{t^t}{e^t}$ and $\frac{(t+1)^t}{t^t}\leq e$. For $k=o(n^{\eps/5})$, the last expression is decreasing for $t\leq k$ and bounding each term by the bound for the $t=1$ term, we have
\eq{\sum_{t}S_tK^{2t}\ll k^7n^{k-1-\eps}=o(n^{k-1})} for $k=o(n^{\eps/7})$.

\end{proof}

\section{Proof of Theorem ~\ref{delocthm}}\label{secthm}

\begin{proof}

We will work on the event 
\eq{E=E_n=\{\rho(X)<1+\epsilon,\lam>1+2\epsilon\text{ for all }\lam\in\bigcup_{\theta\in\Theta}\Lambda^{\theta}\}} which occurs w.h.p. Fix $\theta\in\Theta$ and for $\lam\in\Lambda^{\theta}$, let 
\eq{R_{\lam}:=\left(\frac{X}{\sqrt{n}}-\lam\right)^{-1}} denote the resolvent of $X/\sqrt{n}$. On $E$, $\lam>\rho(X)$, so we may expand $R_{\lam}$ as a Neumann series
\al{R_{\lam}&=-\frac{1}{\lam}\left(1+\frac{1}{\sqrt{n}}\sum_{i\geq 1}\frac{X^i}{n^{(i-1)/2}\lam^i}\right)\\
&=:-\frac{1}{\lam}\left(1+\frac{1}{\sqrt{n}}S_{\lam}\right).}
We write the Jordan decomposition of $A$ as $A=VJU^*$ where $V$ (resp. $U^*$) is the $n\times \rk(A)$ (resp. $\rk(A)\times n$) matrix of generalized right (resp. left) eigenvectors of $A$ associated to nonzero eigenvalues of $A$ satisfying $U^*V=1$ and $J$ is the Jordan matrix of $A$ restricted to nonzero eigenvalues with size $\rk(A)\times\rk(A)$. Starting with the eigenvalue equation $\det(\frac{X}{\sqrt{n}}+A-\lam)=0$ and using the determinant identity $\det(1+AB)=\det(1+BA)$, we have
\al{\det\left(\frac{X}{\sqrt{n}}+A-\lam\right)=0&\Rightarrow\det\left(1+R_{\lam}A\right)=0\\
&\Rightarrow \det\left(1-\frac{1}{\lam}U^*\left(1+\frac{1}{\sqrt{n}}S_{\lam}\right)VJ\right)=0\\
&\Rightarrow \det\left(-\lam+J+\frac{1}{\sqrt{n}}U^*S_{\lam}VJ\right)=0.}
 
Let $J_{\theta}$ be the block matrix of $J$ corresponding to eigenvalue $\theta$ and let $U^*_{\theta}$ and $V_{\theta}$ be the restrictions of $U^*$ and $V$ to the generalized left and right eigenvectors of $\theta$ respectively. Recall Proposition~\ref{detpert} as well as the notation used therein. We apply Proposition~\ref{detpert} with $M=J_{\theta}$ and $P=P^{\theta}=\frac{1}{\sqrt{n}}U^*_{\theta}S_{\lam}V_{\theta}J_{\theta}$.

First note that for each column indexed by $t\in I^{\theta}_v$, $J_{\theta}e_t=\theta e_t$, where $e_t$ is the coordinate vector corresponding to $t$. Hence for $s\in I^{\theta}_u$ and $t\in I^{\theta}_v$, 
\eq{P^{\theta}_{st}=\frac{1}{\sqrt{n}}\theta u^*_sS_{\lam}v_t.}  Observe that the moment assumption made in Theorem ~\ref{delocthm} guarantees the applicability of Lemma~\ref{Sclt} to the collection
\eq{\{\sqrt{n}P^{\theta}_{st}:s\in I^{\theta}_u,t\in I^{\theta}_v,\theta\in\Theta\}.} 

By Lemma ~\ref{Sclt}, $(\sqrt{n}P^{\theta}_{st})_{s,t,\theta}\Rightarrow (F_r)_{r\in I_2}$ defined by $\eqref{Ftheta}$, $\eqref{grdef}$ and $\eqref{gtheta}$. Finally, applying Proposition~\ref{detpert} yields the procedure to determine the fluctuations as specified in Theorem ~\ref{delocthm}.
\end{proof}

\appendix

\section{}\label{secdet}
In this section we state the deterministic perturbation result referred to in the proof of Theorem ~\ref{delocthm}. It is originally attributed to Lidskii. See \cite{lidskii} and references cited within. We remind the reader that the Schur complement of $A$ in the block matrix $\block{cc}{A&B\\C&D}$ is $D-CA^{-1}B$.
\begin{prop}\label{detpert}
Let $M$ be a $d\times d$ deterministic matrix in Jordan form. For notational simplicity, we will assume $M$ has a single eigenvalue $\theta$. Let
\eq{J_k:=\block{cccc}{\theta&1&&\\&\theta&1&\\&& \ddots&1\\&&&\theta}} denote the $k\times k$ Jordan block and write 
\eq{M=\bigoplus_{k=1}^KJ_{k}^{\oplus m_{k}}}  Hence for each $k\in [K]$, $M$ has $m_{k}$ Jordan blocks $J_{k}$.  Let $P_n$ be a sequence of $d\times d$ perturbation matrices with entries of size $o(1)$. Then  $M+P_n$ has spectrum \eq{\Lambda(M+P_n)=\{\lam_{k,m,i}:k\in [K],m\in [m_{k}],i\in[k]\}} with $\lam_{k,m,i}\rightarrow\theta$ for all $k\in K$, $m\in [m_k]$ and $i\in[k]$. The fluctuations \eq{f_{k,m,i}:=\lam_{k,m,i}-\theta} are given by the following procedure. 

Let $c_k:=\sum_{j=1}^km_j$ and set $c:=c_K$. Decompose $P=P_n$ into $c^2$ blocks $(B_{ij})_{i,j=1}^c$ with the $c$ diagonal blocks $(B_{i,i})_{i=1}^c$ having sizes 
\eq{1,\ldots,1,2.\ldots,2,\ldots,K,\ldots K} with $k$ occurring with multiplicity $m_k$. Let $k_i\times k_j$ denote the size of block $B_{i,j}$. This block decomposition is conformal with that of $M$ induced by the $J_k$'s. Let $R=R_n$ be the submatrix of $P$ of size $c\times c$ with entries given by 
\eq{R_{ij}=(B_{ij})_{k_i1}.} Hence $R$ is formed from the lower left elements of the blocks in the decomposition of $P$. 

Let $E_k=R_{c_k\times c_k}$ be upper left submatrices of $R$ and let $F_k$ be the $m_k\times m_k$ Schur complement of $E_{k-1}$ in $E_{k}$, where we set $F_1:=E_1$. Then, to leading order, the fluctuations $f_{k,m,i}$ are given by the $k$ $k$-th roots of the $m_k$ eigenvalues of $F_k$ for each $k\in [K]$. If $M$ has multiple eigenvalues, we apply the above procedure to each eigenvalue separately.
\end{prop}

We remark on a few special cases of Proposition ~\ref{detpert}. We denote the entries of $P=P_n$ by $p_{ij}$ and assume $p=O(\frac{1}{\sqrt{n}})$ (as will turn out to be the case in our applications).

\begin{enumerate}
\item Suppose $M = \diag(\theta_1,\ldots,\theta_d)$ is diagonal with distinct eigenvalues. Let $\lam_j$ denote the corresponding eigenvalues of $M+P$ in the sense that $\lam_j\rightarrow\theta_j$ as $n\rightarrow\infty$ Then 
\eq{f_j:=\lam_j-\theta_j=p_{jj}(1+o(1)).}
\item Suppose $M=\theta I_d$. Then $\{\sqrt{n}(\lam_j-\theta)\}_{j=1}^d$ converge to the $d$ eigenvalues of $\sqrt{n}P$.
\item Suppose $M=J_d(\theta)$. Then $\{n^{\frac{1}{2d}}(\lam_j-\theta)\}_{j=1}^d$ converge to the $d$ roots of $\sqrt{n}P_{k1}$.
\end{enumerate}

\section{}\label{sec4eps}

In this appendix, we extend the results involving the moment method, namely Proposition ~\ref{myclt} and Lemma ~\ref{Sclt} using a truncation argument (see \cite{bai-silver}). Consider the following two assumptions on the atom distribution $x$.
\begin{enumerate}[label=(\roman*)]
\item \label{Mhyp}$|x|\leq K=O(n^M)$.
\item \label{mhyp}$\E|x|^{m}<\infty$, $m=2/M$.
\end{enumerate}
We show that if Proposition ~\ref{myclt} and Lemma ~\ref{pathlemma} hold for \ref{Mhyp} with $M< 1/2$, then they hold for \ref{mhyp}. 

Suppose we have \ref{mhyp} with $m>4$, corresponding to $M=2/m<1/2$. We first show that the event
\eq{\{|x_{ij}|\leq n^{M} \text{ for all } i,j\in[n]\}} occurs w.h.p. Indeed, we have
\all{\label{ep1}\P\left[|x_{ij}|\geq n^{M} \text{ some } i,j\in [n]\right]
&\leq n^2\P\left[|x|^{m}\geq n^2\right].}
Since $n^2\mathbbm{1}_{|x|^m\geq n^2}\leq |x|^m$ and $\E|x|^m<\infty$, the last expression converges to $0$ by the dominated convergence theorem.  

Now define the truncated random variables $\hat{x}:=x\mathbbm{1}_{|x|\leq n^{M}}$ and $\hat{X}=(\hat{X})_{ij}$ by $\hat{X}_{ij}:=\hat{x}_{ij}$. 
While $\hat{x}$ is bounded, it no longer has mean zero. On the other hand, for $n$ sufficiently large, we have
\all{|\E \hat{x}|&\leq \E|x\mathbbm{1}_{|x|\geq n^{M}}|\notag\\
&\leq \frac{\E|x|^{m}\mathbbm{1}_{|x|\geq n^{M}}}{n^{(m-1)M}}\notag\\
&\ll n^{-(m-1)M}\notag\\
&\leq n^{-3/2}\label{xhat}.} By Schur's test for the operator norm of a matrix, we have
\eql{\label{schur}\|\E\hat{X}\|=O(n^{-1/2}).} Now let $\tilde{x}:=\hat{x}-\E\hat{x}$ and $\tilde{X}:=\hat{X}-\E\hat{X}$ denote the truncated and centered random variables. By construction, $\E\tilde{x}=0$. Furthermore,
\eql{\label{tildexvar}\E|\tilde{x}|^2=\E|\hat{x}|^2-|\E\hat{x}|^2\rightarrow\E|x|^2=1} by ~\eqref{xhat} and dominated convergence. Given ~\eqref{tildexvar}, it is easy to check that under ~\ref{Mhyp}, Proposition ~\ref{myclt} is valid for $\tilde{X}$. Since $\E|\tilde{x}|^2\leq\E|x|^2$, Lemma ~\ref{Sclt} also valid for $\tilde{x}$. 
To prove the validity of Proposition \ref{myclt} and Lemma ~\ref{Sclt} for $x$ under ~\ref{mhyp}, it suffices to prove the following.
\begin{lem}\label{truncation} Suppose $u=u_n$ and $v=v_n$ are unit vectors in $\C^n$.
Then for every $\gamma>0$, the event
\eq{A_{n,\gamma}:=\bigcup_{k\leq\log^2n}\{|u^*\hat{X}^kv-u^*\tilde{X}^kv|>\gamma n^{(k-1)/2}\}} occurs w.h.p.
\end{lem}

We first state a result that is a consequence of the proof in \cite{bai-yin}. Following the notation of \cite{bai-yin} we define $\delta:=n^{M-1/2}$ so that $|\tilde{x}|\leq\delta\sqrt{n}$. Fix $z>k+1$ and $p$ a positive integer.
Then \al{\P\left[\left\|\left(\frac{1}{\sqrt{n}}{\tilde{X}}\right)^k\right\|\geq z\right]&\leq z^{-2p}n^{-pk}\E\Tr\left(\tilde{X}^k\left(\tilde{X}^k\right)^*\right)^p\\
&=:z^{-2p}n^{-pk}E_n.}
In \cite{bai-yin}(pg. $561$), it is shown that
\eq{E_n\leq n^{kp+\frac{3}{2}}\sum_{l=1}^{pk}\binom{2kp}{2l}(k+1)^{2kp-2l+2p}(2kp)\left(\frac{6kp\delta^{1/6}}{\log\frac{\delta\sqrt{n}}{(2kp)^3}}\right)^{6kp-6l}\delta^{kp-l}.} 

In our application, $k\leq\log^2n$ and choosing $p=\delta^{-1/7}$ say, we have
\eql{\label{baifrac}\frac{6kp\delta^{1/6}}{\log\frac{\delta\sqrt{n}}{(2kp)^3}}\rightarrow 0.} In fact, the left-hand side of ~\eqref{baifrac} is less than $1$ for $n\geq N(m)$.

For such $n$, following \cite{bai-yin}(pg. $562$), it then follows that 
\eq{z^{-2p}n^{-pk}E_n\leq \left((2kpn^2)^{1/p}(1+(k+1)\delta^{1/2})^{2k}\left(\frac{k+1}{z}\right)^2\right)^p.}

Choosing $z=3k$ say, for any $k\leq\log^2n$ we have
\eql{\label{bybound}\P\left[\left\|\left(\frac{1}{\sqrt{n}}{\tilde{X}}\right)^k\right\|\geq 3k\right]=O(e^{-n^c})} for some $c=c(m)>0$. We now turn to the proof of Lemma ~\ref{truncation}.
\begin{proof}

By ~\eqref{bybound}, we may assume $\left\|\left(\frac{1}{\sqrt{n}}\tilde{X}\right)^k\right\|\leq 3k$ for all $k\leq \log^2 n$ which occurs w.h.p. We will need the crude bound
\eql{\label{crudesum}\sum_{\substack{a_1+\ldots+a_k=n\\ a_i\geq 0}}\prod_{i=1}^ka_i\leq n^{2k}.} 
We then have 
\al{\frac{1}{n^{(k-1)/2}}|u^*\hat{X}^kv-u^*\tilde{X}^kv|&\leq n^{-(k-1)/2}\|(\tilde{X}+\E\hat{X})^k-\tilde{X}^k\|\\
&\leq\sum_{l=1}^k\frac{1}{n^{(l-1)/2}}\sum_{l'=0}^{l+1}\sum_{a_1+\ldots+a_{l'}=k-l}\prod_{i=1}^{l'}\left\|\left(\frac{\tilde{X}}{\sqrt{n}}\right)^{a_i}\right\|\|\E\tilde{X}\|^l\\
&\leq\sum_{k=1}^l \frac{1}{n^{l-1/2}}\sum_{l'=0}^{l+1}(3k^2)^{l'}\\
&\leq\sum_{l=1}^k\frac{l+1}{n^{(l-1)/2}}\left(\frac{3k^2}{\sqrt{n}}\right)^{l+1}=o(1),} where we have used \eqref{crudesum} and ~\eqref{schur} in the third line. 

\end{proof}

\bibliographystyle{plain}
\bibliography{refs}
\end{document}